\documentclass[11pt,a4paper]{article}
\usepackage{amsmath,amssymb,amsthm,amscd,mathrsfs}
\usepackage[top=25mm, bottom=30mm, left=30mm, right=30mm]{geometry}
\usepackage{color}
\usepackage{comment}
\usepackage{bm}
\usepackage[all]{xy}

\newtheorem{df}{\bf Definition}[section]
\newtheorem{thm}[df]{\bf Theorem}
\newtheorem{lem}[df]{\bf Lemma}

\newtheorem{prop}[df]{\bf Proposition}

\newtheorem{notation}[df]{\bf Notation}

\newtheorem*{claim}{\bf Claim}

\newtheorem{thmA}{\bf Theorem}
\newtheorem{corA}[thmA]{\bf Corollary}

\newtheorem{propA}[thmA]{\bf Proposition}

\newcommand{\R}{\mathbb{R}}
\newcommand{\C}{\mathbb{C}}
\newcommand{\Z}{\mathbb{Z}}

\newcommand{\N}{\mathbb{N}}
\newcommand{\B}{\mathbb{B}}
\newcommand{\K}{\mathbb{K}}
\newcommand{\M}{\mathbb{M}}

\newcommand{\frakS}{\mathfrak{S}}

\newcommand{\cK}{\mathcal{K}}

\newcommand{\cF}{\mathcal{F}}

\newcommand{\cC}{\mathcal{C}}

\newcommand{\cI}{\mathcal{I}}
\newcommand{\cW}{\mathcal{W}}

\newcommand{\act}{\curvearrowright}

\newcommand{\ri}{\mathrm{i}}

\newcommand{\sym}{\mathrm{sym}}
\newcommand{\full}{\mathrm{full}}
\newcommand{\anti}{\mathrm{anti}}

\newcommand{\Ber}{\mathrm{Ber}}

\newcommand{\supp}{\mathrm{supp}}

\newcommand{\sfX}{\mathsf{X}_\full}

\newcommand{\sfZ}{\mathsf{X}_\anti}
\newcommand{\sfW}{\mathsf{X}_\Ber}

\newcommand{\Ad}{\operatorname{Ad}}
\newcommand{\Prob}{\operatorname{Prob}}
\newcommand{\id}{\text{\rm id}}

\newcommand{\Aut}{\operatorname{Aut}}

\newcommand{\otm}{\otimes_{\rm min}}
\newcommand{\ota}{\otimes_{\rm alg}}

\title{\bf Boundary and rigidity of nonsingular \\Bernoulli actions}
\author{Kei Hasegawa
\and Yusuke Isono\thanks{Research Institute for Mathematical Sciences, Kyoto University, 606-8502, Kyoto, Japan \protect \\  E-mail: \texttt{isono@kurims.kyoto-u.ac.jp} \protect \\  
YI is supported by JSPS KAKENHI Grant Number 20K14324.}
\and Tomohiro Kanda}
\date{}
\begin{document}
\maketitle

\begin{abstract}
	Let $ G $ be a countable discrete group and consider a nonsingular Bernoulli shift action  $ G  \act \prod_{g\in  G }(\{0,1\},\mu_g)$ with two base points. When $ G $ is exact, under a certain finiteness assumption on the measures $\{\mu_g\}_{g\in  G }$, we construct a boundary for the Bernoulli crossed product C$^*$-algebra that admits some commutativity and amenability in the sense of Ozawa's bi-exactness. As a consequence, we obtain that any such Bernoulli action  is solid. This generalizes solidity of measure preserving Bernoulli actions by Ozawa and Chifan--Ioana, and is the first rigidity result in the non measure preserving case. For the proof, we use anti-symmetric Fock spaces and left creation operators to construct the boundary and therefore the assumption of having two base points is crucial.
\end{abstract}

\section{Introduction and main results}

Rigidity in ergodic theory and von Neumann algebra theory has been intensively studied in the last two decades. The most well studied class is the \textit{Bernoulli shift actions}: actions of countable discrete groups $ G $ on product spaces $\prod_{ G }(X_0,\mu_0)$ given by $(g\cdot x)_h = x_{g^{-1}\cdot h}$. It provides a wide range of rigidity phenomenon and is regarded as a very useful source in the theories. 
A typical rigidity of the Bernoulli actions is \textit{solidity} \cite{Oz04,CI08}. This means, if we denote by $\mathcal R$ the associated orbit equivalence relation of a given Bernoulli action, then for any subequivalence relation $\mathcal S \subset \mathcal R$, there exists a partition $\{X_n\}_{n\geq 0}$ of the product space into $\mathcal S$-invariant measurable sets such that $\mathcal S|_{X_0}$ is hyperfinite and  $\mathcal S|_{X_n}$ is strongly ergodic for all $n\geq 1$. In particular any ergodic non-hyperfinite subequvalence relation of $\mathcal R$ is strongly ergodic. When $ G $ is non-amenable, this strengthens  strong ergodicity of the Bernoulli action and is rigidity for subequivalence relations.

By simply exchanging the probability measure $\mu_0$ with a family $\{\mu_g\}_{g\in  G }$, combining with Kakutani's criterion for infinite product measures \cite{Ka48}, one can define a \textit{nonsingular Bernoulli action} $ G  \act \prod_{g\in  G } (X_0,\mu_g)$. 
It has been studied for years when $ G  =\Z$ and many interesting results on ergodicity and type classifications were obtained, see for example \cite{Kr70,Ha81,Ko09,BKV19}. Very recently, the study of the general case was initiated by Vaes and Wahl \cite{VW17}. It was proved in \cite{VW17,BKV19} that a group $ G $ admits an ergodic nonsingular Bernoulli action of type III if and only if $ G $ has a (nontrivial) cocycle for the left regular representation. Therefore the existence of type III Bernoulli actions for a given group is closely related to \textit{geometry} of the group. From this point of view, it is a natural question to ask if such a type III Bernoulli action admits some rigidity aspects when $ G $ has a nice geometric property (in the sense of cocycles).

The \textit{pmp} (probability measure preserving) Bernoulli actions has been studied in two independent ways: one is by \textit{Popa's deformation/rigidity theory} (see  surveys \cite{Po06,Va10,Io17}) and the other is by \textit{Ozawa's boundary amenability} \cite{Oz06,BO08}. For this reason, the aforementioned solidity of the pmp Bernoulli actions has two independent proofs by Chifan--Ioana \cite{CI08} and by Ozawa \cite{Oz04}. 
Both proofs, however, heavily depend on the measure preserving condition and seem to be not working for the non-pmp case. 
Further, all known examples of solid actions (see \cite{Oz08,Ma16}) depend on similar ideas which use the measure preserving condition. 
Thus solidity of non-pmp Bernoulli actions is still open because of this technical issue.

In the present paper, we deal with this problem. In our main theorem below, we prove solidity for a wide class of nonsingular and non-pmp Bernoulli actions, which is the first rigidity results for non-pmp Bernoulli actions. 

\begin{thmA}\label{thmA}
	Let $ G $ be a countable discrete group and consider a product measure space with two base points
	$$ (\Omega , \mu):=\prod_{g\in G} (\{0,1\}, p_g\delta_0 + q_g \delta_1)\quad \text{where $p_g\in (0,1)$ and $p_g+q_g=1$ for all $g\in G$}. $$
Assume that $(\Omega,\mu)$ has no atoms and satisfies Kakutani's condition, so that the nonsingular Bernoulli action $ G  \act (\Omega,\mu)$ is defined. Assume further that
\begin{itemize}
	\item[$\rm (i)$] $ G $ is exact;
	\item[$\rm (ii)$] for any $g\in  G $, $p_h = p_{gh}$ for all but finitely many $h\in G$. 
\end{itemize}
Then the Bernoulli action $ G  \act (\Omega,\mu)$ is solid.
\end{thmA}

We mention that Theorem \ref{thmA} actually holds for some  \textit{generalized} Bernoulli actions, see Theorem \ref{main thm} and its assumptions (i) and (ii) in the statement.

For the proof of the main theorem, we use Ozawa's boundary argument in a very different context, see the explanation below. 
Assumptions (i) and (ii) are crucial in our proof, although we believe that there are more appropriate (and weaker) assumptions. Assumption (ii) should be understood as a condition for the associated cocycle.

Below we give two concrete examples satisfying the assumptions in Theorem \ref{thmA}. Recall that an essentially free action $ G  \act (X,\mu)$ is \textit{amenable in the sense of Zimmer} if the associated orbit equivalence relation is hyperfinite. 

\subsection*{Example 1: Groups with more than one end}

Recall that a countable discrete group $ G $ has \textit{more than one end} if it has a (nontrivial) \textit{almost invariant subset}, that is, there is a subset $W\subset  G $ such that $|W|=|W^c|=\infty$ and $|W\triangle gW|<\infty$ for all $g\in  G $. Any such $W$ gives rise to a cocycle for the left regular representation, which was already used in \cite{BKV19}. Any (nontrivial) free product and any HNN extension (with finite subgroup) have more than one end, see \cite[Remark 6.7]{BKV19}. Combining Theorem \ref{thmA} with (the proof of) \cite[Theorem 6.9.1]{BKV19}, we get the following corollary. 

\begin{corA}\label{corB}
	For any countable non-amenable exact group $ G $ with more than one end, there exists a nonsingular Bernoulli action with two base points such that it is essentially free, ergodic of stable type $\rm III_1$, non-amenable in the sense of Zimmer, and solid. 
\end{corA}

\subsection*{Example 2: Groups acting on trees}

	Let $\mathcal T$ be a (locally finite) tree and we denote by $E(\mathcal T)$ the set of all the edges. Let $ G  \leq \Aut(\mathcal T)$ be a countable discrete subgroup such that $ G $ is non-amenable and $\mathcal T/ G $ is a finite set. Then following \cite[Section 10]{AIM19}, for any fixed $p\in (0,1)$, there is a probability measure $\mu_p$ on the set $\Omega$ of all possible orientations of $E(\mathcal T)$. The natural $ G $-action on $\Omega$ induces a nonsingular action whose ergodic properties were studied in \cite[Theorem 10.5]{AIM19}. It is straightforward to see that the action is essentially free if $\mathcal T$ has no leaves (here a \textit{leaf} means a vertex that has exactly one edge). As mentioned in  \cite[Remark 10.1]{AIM19}, up to exchanging $ G $ with an index 2 subgroup $ G ^+$ if necessary, one can identify $ G ^+ \act (\Omega,\mu_p)$ with a nonsingular (generalized) Bernoulli action. 
Since it satisfies all the assumptions in Theorem \ref{thmA} (see assumptions (i) and (ii) in Theorem \ref{main thm}), the action is solid. Since $ G ^+$ has index 2 in $ G $, this implies solidity of the $ G $-action. We get the following corollary. (For the non-amenability part below, see the proof of \cite[Theorem 6.9.1]{BKV19} and use \cite[Proposition 5.3]{VW17} in a similar way.)

\begin{corA}\label{corC}
	Keep the setting. For any $p\in (0,1)$ which is sufficiently close to $1/2$ (but $p\neq 1/2$), the action $ G  \act (\Omega,\mu_p)$ is essentially free, ergodic of type $\rm III_\lambda$, where $\lambda=\min\{\frac{1-p}{p},\frac{p}{1-p}\}$, non-amenable in the sense of Zimmer, and solid. 
\end{corA}

Theorem \ref{thmA} also has an application to rigidity of associated von Neumann algebras. Recall that a diffuse von Neumann algebra $M$ (with separable predual) is \textit{solid} \cite{Oz03} if for any diffuse von Neumann subalgebra $A\subset M$ which is a range of a faithful normal conditional expectation (say, \textit{with expectation}), $A'\cap M$ is hyperfinite. 
For bi-exactness of groups below, see Subsection \ref{Bi-exactness and condition (AO)}. We obtain new examples of solid type III factors.

\begin{corA}\label{corD}
	Let $ G  \act (\Omega,\mu)$ be an action given in Theorem \ref{thmA}, Corollary \ref{corB} or Corollary \ref{corC}. If $ G $ is bi-exact, then the associated von Neumann algebra $L^\infty(\Omega,\mu)\rtimes  G $ is solid.
\end{corA}

We note that any free group $\mathbb F_n$ $(2\leq n<\infty)$ is exact, bi-exact, with more than one end, and acting on its Cayley tree. Hence one can easily construct nonsingular (generalized) Bernoulli actions whose von Neumann algebras are solid.

\subsection*{Strategy for the proof}

We explain our strategy for the proof of the main theorem.  \textit{Anti-symmetric Fock spaces} play a crucial role in the proof. Let $ G  \act (\Omega,\mu)$ be an action given in Theorem \ref{thmA} (or Theorem \ref{main thm}).

Firstly, we embed $L^\infty(\Omega,\mu)$ into an Araki--Woods factor (infinite tensor product of 2 by 2 matrices) in a diagonal way. Then we identify the factor with a von Neumann algebra generated by some left creation operators acting on an anti-symmetric Fock space $\cF_\anti$. Consequently the crossed product $L^\infty(\Omega,\mu)\rtimes  G $ acts on $\cF_\anti\otimes \ell^2( G )$. See Subsection \ref{Non-singular Bernoulli actions via the CAR construction} for this first step. 

Secondly, we develop a boundary theory for creation operators on $\cF_\anti$. By assumption (i), we prove amenability of a $ G $-action on the boundary. This part is inspired by Ozawa's proof of solidity of pmp Bernoulli actions \cite{Oz04}. See Section \ref{Boundary amenability on anti-symmetric Fock spaces} for the second step.

Thirdly, we use the boundary to obtain a condition (AO) phenomena for the reduce crossed product C$^*$-algebra $C(\Omega)\rtimes_{\rm red} G $, see Proposition \ref{propE} below. Once we get it, the proof in \cite{Oz04} is applied and we obtain the desired solidity. However to prove Proposition \ref{propE} is the most technical part in the proof and we need lots of computations. This difficulty comes from the fact that the action does not preserve the measure, hence this part is the main step for the proof. We need assumption (ii) in this procedure. See Subsection \ref{How to use the boundary}, \ref{Proof of Theorem A}, and \ref{Proof of key lemma} for this last step.

Below we introduce \textit{condition AO} (see Subsection \ref{Bi-exactness and condition (AO)}) in our setting. This should be compared to  \cite[Proposition 4.3]{Oz04}. The map $J$ in the proposition is the modular conjugation (see Subsection \ref{Standard representations}) and $\{e_{g,g}\}_{g\in  G }$ is the matrix unit in $\B(\ell^2( G ))$.

\begin{propA}\label{propE}
	Keep the setting from Theorem \ref{thmA}. Put $A:=C(\Omega)\rtimes_{\rm red } G $ and 
	$$ \cK:=\left\{\sum_{g\in  G }x_g\otimes e_{g,g}\in \B(L^2(\Omega,\mu) \otimes \ell^2( G )) \;\middle|\; x_g\in \K(L^2(\Omega,\mu)) \ \text{for all }g\in  G  \right\} .$$
We denote by $\mathrm{M}(\cK)$ the multiplier algebra of $\cK$. 
Then the natural $\ast$-homomorphism 
\begin{align*}
	\nu\colon A \ota JAJ &\to \mathrm{M}(\cK)/\cK,\quad \nu(x\otimes JyJ) = x JyJ + \cK
\end{align*}
is well defined and bounded with respect to the minimal tensor norm.
\end{propA}

Before finishing the introduction, we mention the following two remarks. Firstly, we can actually construct a boundary which acts on $L^2(\Omega,\mu)$ (rather than $\cF_\anti$). In this case, we can forget all information of $\cF_\anti$, see Subsection \ref{Boundary for Bernoulli actions}. 
Secondly, Proposition \ref{propE} actually holds at the level of the Araki--Woods factor, see Proposition \ref{prop-AO}. In particular, in Theorem \ref{solid-Bernoulli-thm}, we get examples of solid actions on Araki--Woods factors. For this, we prepare a general theory about solidity for actions in Section \ref{Solidity for actions on von Neumann algebras}.

\bigskip
\noindent
{\bf Acknowledgment. }
We would like to thank C. Houdayer and A. Marrakchi for useful comments for the first draft of the article.

\tableofcontents

\section{Preliminaries}

\subsection{Nonsingular Bernoulli actions}
\label{Nonsingular Bernoulli actions}

	Let $ G $ be a countable discrete group acting on a countable set $\cI$. Consider a product measure space with two base points
	$$ (\Omega , \mu):=\prod_{i\in \mathcal I} (\{0,1\}, p_i\delta_0 + q_i \delta_1)\quad \text{where $p_i\in (0,1)$ and $p_i+q_i=1$ for all $i\in \cI$}. $$
Kakutani's criterion \cite{Ka48} shows that, for any $g\in G$, the assignment $g\cdot (x_i)_i := (x_{g^{-1}\cdot i})_i$ for $(x_i)_i\in \Omega$ defines a nonsingular automorphism of $(\Omega,\mu)$ if and only if
	$$ \sum_{i\in I} \left( \sqrt{\mu_i(0)} - \sqrt{\mu_{g\cdot i}(0)} \right)^2+ \left( \sqrt{\mu_i(1)} - \sqrt{\mu_{g\cdot i}(1)} \right)^2<\infty.$$
It is straightforward to see that $(\Omega,\mu)$ has no atoms if and only if 
	$$\sum_{i\in I} \min\{\mu_i(0),\mu_i(1)\} = \infty.$$
Both conditions are trivially satisfied for our concrete examples in Corollary \ref{corB} and \ref{corC} and hence nonsingular (generalized) Bernoulli actions are defined.

\subsection{Bi-exactness and condition (AO)}
\label{Bi-exactness and condition (AO)}

We refer the reader to \cite[Subsection 4.4]{BO08} for amenability of actions, \cite[Section 5]{BO08} for exactness, and \cite[Section 15]{BO08} for bi-exactness of groups.

Let $ G  \act X$ be an action of a discrete group on a compact Hausdorff space. Consider the associated action $ G  \act C(X)$. 
We say that it is \textit{(topologically) amenable} if the reduced and full crossed products coincide: $C(X)\rtimes_{\rm red} G  = C(X)\rtimes_{\rm full}  G $. 
More generally for any action $ G  \act A$ on a unital C$^*$-algebra $A$, the action is \textit{amenable} if the restriction to the center of $A$ is amenable. In that case, we also have $A\rtimes_{\rm red} G  = A\rtimes_{\rm full}  G $.

We say that a countable discrete group $ G $ is \textit{bi-exact} if the left and right translation action $ G  \times  G  \act \ell^\infty( G )/c_0( G )$ is amenable. This should be compared to the fact that $ G $ is \textit{exact} if and only if the left translation $ G  \act \ell^\infty( G )/c_0( G )$ is amenable. It is proved in \cite[Lemma 15.1.4]{BO08} that if $ G  $ is bi-exact, then the following algebraic $\ast$-homomorphism
	$$ \nu\colon C_\lambda^*( G )\ota C_\rho^*( G )\to \B(\ell^2( G ))/\K(\ell^2( G ));\quad a\otimes b\mapsto ab+\K(\ell^2( G )) $$
is bounded with respect to the minimal tensor norm (say, \textit{min-bounded}). This boundedness is called \textit{condition (AO)} \cite{Oz03}. In this article, we study such a condition (AO) phenomena for nonsingular Bernoulli actions.

\subsection{Standard representations}
\label{Standard representations}

For Tomita--Takesaki's modular theory, we refer the reader to \cite{Ta03}. Let $M$ be a von Neumann algebra equipped with a faithful normal state $\varphi$. Consider the GNS representation $L^2(M):=L^2(M,\varphi)$. The involution $M\ni a \mapsto a^*\in M$ induces a closed operator $S_\varphi$ on $L^2(M)$ and we denote by $S_\varphi = J_\varphi \Delta_\varphi $ its polar decomposition. We say that $\Delta_\varphi$ and $J_\varphi$ the \textit{modular operator} and the \textit{modular conjugation}. The natural $M$-action on $L^2(M)$ is called \textit{Haagerup's standard representation} (e.g.\ \cite[Chapter IX $\S$2]{Ta03}) and the associated $J$-map is $J_\varphi$. We regard $J_\varphi M J_\varphi = M'$ as a right representation of $M$ on $L^2(M)$. 
For any $\alpha\in \Aut(M)$, there is a unique unitary $U\in \B(L^2(M))$ such that $\alpha = \Ad(U)$ and that $U$ preserves the structure of the standard representation (e.g.\ $UJ_\varphi = J_\varphi U$). We call $U$ the \textit{standard implementation} of $\alpha$.

Let $\alpha\colon  G \act M$ be an action of a discrete group $ G $ and $U_g \in \mathcal U(L^2(M))$ the standard implementation of $\alpha_g$ for each $g\in  G $. We use the standard representation $L^2(M)\otimes \ell^2( G )$ of $M\rtimes_{\alpha}  G $ given by: with the $J$-map $J_M$ on $L^2(M)$,
\begin{align*}
	\text{left action}\quad &G  \ni g \mapsto U_g \otimes \lambda_g;\quad M\ni a\mapsto a\otimes 1_ G ;
\\
	\text{right action}\quad &  G  \ni g \mapsto 1\otimes \rho_g;\quad J_MMJ_M \ni a \mapsto \pi_r (a)=\sum_{h\in  G }U_h a U_{h}^{-1}\otimes e_{h,h};\\
	\text{$J$-map}\quad 	&J= \sum_{h\in  G }U_hJ_M \otimes e_{h,h^{-1}}.
\end{align*}
Here $\{e_{g,h}\}_{g,h\in  G }$ is the matrix unit in $\B(\ell^2( G ))$. Note that the commutant is regarded as a crossed product by 
	$$(M\rtimes_{\alpha}  G )'= J(M\rtimes_{\alpha}  G )J  =  (J_M M J_M)\rtimes_{\alpha^r}  G  = M'\rtimes_{\alpha^r} G ,$$
where $\alpha^r\in \Aut(M')$ is given by $\Ad(U_g)$.

\subsection{Fock spaces and associated von Neumann algebras}
\label{Fock spaces and associated von Neumann algebras}

For Fock spaces, we refer the reader to \cite[Subsection 5.2]{BR97}. 

Let $H$ be a Hilbert space. For each $n\in \N$, consider the $n$-tensor product of $H$, denoted by $H^{\otimes n}$. We define the \textit{full Fock space} $\cF_\full(H)$ by the completion of 
	$$\C \Omega \oplus \bigoplus_{n \geq 1} H^{\otimes n}.$$
We denote by $\frakS_n$ the symmetric group over $\{1,\ldots,n\}$. Put
	$$ \pi^{(n)}_\sigma (\xi_1\otimes \cdots \otimes \xi_n) := \xi_{\sigma(1)}\otimes \cdots \otimes \xi_{\sigma(n)},\quad \text{for }\sigma\in \frakS_n, \ \xi_1,\ldots ,\xi_n\in H $$
which defines a unitary representation $\pi^{(n)}$ of $\frakS_n$ on $H^{\otimes n}$. Put
 	$$P_-^{(n)}:=\sum_{\sigma\in \mathfrak{S}_n} \mathrm{sgn}(\sigma)\pi^{(n)}_\sigma\in \B(H^{\otimes n})$$
and observe that $\frac{1}{n!}P_-^{(n)}$ is a projection. We use the notation
	$$\xi_1\wedge \cdots \wedge \xi_n:= \frac{1}{n!}P_-^{(n)}(\xi_1\otimes \cdots \otimes \xi_n) =\frac{1}{n!}\sum_{\sigma\in \frakS_n}\mathrm{sgn}(\sigma)(\xi_{\sigma(1)}\otimes \cdots \otimes \xi_{\sigma(n)})$$
for all $\xi_1,\ldots ,\xi_n\in H$. It holds that
	 $$\xi_{\sigma(1)}\wedge \cdots \wedge \xi_{\sigma(n)}=\mathrm{sgn}(\sigma)(\xi_1\wedge \cdots \wedge \xi_n),\quad  \text{for all }\sigma\in \frakS_n.$$
In particular $\xi_1\wedge \cdots \wedge \xi_n=0$ if $\xi_i=\xi_j$ for some $i\neq j$. 
We denote the range of $\frac{1}{n!}P_-^{(n)}$ by $H^{\wedge n}$ and note that
	$$H^{\wedge n} = \{\xi\in H^{\otimes n}\mid \pi^{(n)}_\sigma\xi = \mathrm{sgn}(\sigma)\xi\ \text{ for all }\sigma\in \frakS_n\}. $$
Consider a large projection
	$$ P_-:=1_\Omega\oplus \bigoplus_{n\geq 1} \frac{1}{n!}P_-^{(n)} \in \B(\cF_\full(H)).$$
We define the \textit{anti-symmetric Fock space} $\cF_\anti(H)$ by the range of $P_-$, that is,
	$$ \cF_\anti(H):=\C \Omega \oplus \bigoplus_{n\geq 1}H^{\wedge n}. $$
For $\xi\in H$, define \textit{left} and \textit{right creation operators} by
\begin{align*}
	&\ell(\xi)(\xi_1\wedge \cdots\wedge \xi_n) := \sqrt{n+1}\, (\xi\wedge \xi_1\wedge \cdots\wedge \xi_n);\\
	& r(\xi)(\xi_1\wedge \cdots\wedge \xi_n) := \sqrt{n+1}\, (\xi_1\wedge \cdots\wedge \xi_n\wedge \xi)
\end{align*}
for all $n\in \N$ and $\xi_1,\ldots,\xi_n\in H$. It holds that 
\begin{align*}
	\ell(\xi)^*(\xi_1\wedge\cdots \wedge \xi_n)
	&= \frac{1}{\sqrt{n}}\sum_{k=1}^n(-1)^{k-1}\, \langle \xi_k,\xi \rangle\, (\xi_1\wedge\cdots\wedge \check{\xi}_k\wedge \cdots \wedge \xi_n),\\
	r(\xi)^*(\xi_1\wedge\cdots \wedge \xi_n)
	&= \frac{1}{\sqrt{n}}\sum_{k=1}^n(-1)^{k-1}\, \langle \xi_{n-k+1},\xi \rangle\, (\xi_1\wedge\cdots\wedge \check{\xi}_{n-k+1}\wedge \cdots \wedge \xi_n),
\end{align*}
where the check $\check{\xi}_k$ means it is ignored in the tensor. They satisfy the \textit{canonical anti-commutation relation} (CAR)  
\begin{align*}
	\ell(\xi)\ell(\eta) + \ell(\eta)\ell(\xi) = 0,\quad 
	\ell(\xi)\ell(\eta)^* + \ell(\eta)^*\ell(\xi) = \langle \xi,\eta\rangle,\quad \text{for all }\xi,\eta\in H.
\end{align*}
The C$^*$-algebra generated by $\ell(\xi)$ for all $\xi\in H$ is called the \textit{CAR algebra}. It is isomorphic to the UHF algebra of type $2^{d}$, where $d=\dim H$. 

Let $U\colon H\to H$ be a unitary operator. Then for each $n\in \N$, the assignment 
	$$ \xi_1\wedge \cdots \wedge \xi_n \mapsto (U\xi_1)\wedge \cdots \wedge (U\xi_n) $$
defines a unitary on $H^{\wedge}$, which we denote by $U^{\wedge n}$. We can define a large unitary 
	$$ U^\anti:=1_\Omega \oplus \bigoplus_{n\geq 1} U^{\wedge n}  \quad \text{on }\cF_\anti.$$
It satisfies $U^{\anti}\ell(\xi)U^{\anti *}=\ell(U\xi)$ and $U^{\anti}r(\xi)U^{\anti *}=r(U\xi)$ for all $\xi \in H$.

We next introduce an associated von Neumann algebra. For this, consider a \textit{real} Hilbert space $H_\R$ and extend it to a complex Hilbert space by $H:=H_\R\otimes_\R \C$. We identify $H_\R\otimes_\R 1 = H_\R$. Let $I$ be the involution for $H_\R \subset H$, that is, $I(\xi\otimes \lambda)=\xi\otimes \overline{\lambda}$ for $\xi\otimes \lambda\in H_\R \otimes_\R \C$. Note that $H_\R=\{ \xi\in H\mid I\xi = \xi \}$, hence $I$ remembers the real structure $H_\R \subset H$.

Let $U\colon \R\to \mathcal O(H_\R)$ be any strongly continuous representation of $\R$ on $H_\R$. We extend $U$ on $H$ by $U_t\otimes \id_{\C}$ for $t\in \R$ and use the same notation $U_t$. By Stone's theorem, take the infinitesimal generator $A$ satisfying $U_t=A^{\ri t}$ for all $t\in \R$. Consider an embedding
\begin{align*}
	&H \to H;\quad \xi \mapsto \frac{\sqrt{2}}{\sqrt{1+A^{-1}}}\, \xi=:\widehat{\xi}.
\end{align*}
We define associated C$^*$ and von Neumann algebras by
\begin{align*}
	& \mathrm{C}^*_\anti(H_\R,U):= \mathrm{C}^*\{ W(\widehat{\xi}) \mid \xi \in H\};
	&  \Gamma_\anti(H_\R,U):= \mathrm{W}^*\{ W(\widehat{\xi}) \mid \xi \in H\},
\end{align*}
where $W(\widehat{\xi}):=\ell(\widehat{\xi}) + \ell(\widehat{I\xi})^*$. Then $\Omega$ is cyclic and separating, and the vacuum state $\langle \,\cdot\, \Omega,\Omega \rangle$ defines a faithful normal state on $ \Gamma_\anti(H_\R,U)$. The associated modular operator $\Delta$ and the conjugation $J$ on $\cF_\anti(H)$ is given by
\begin{align*}
	 \Delta^{\ri t} (\xi_1\otimes \cdots\otimes \xi_n) &= (U_{-t}\xi_1)\otimes \cdots\otimes (U_{-t}\xi_n)\\
	 J(\xi_1\otimes \cdots\otimes \xi_n) &= (I\xi_n)\otimes \cdots\otimes (I\xi_1)
\end{align*}
for all $t\in \R$, $n\in \N$ and $\xi_1,\ldots,\xi_n\in H$. We refer the reader to \cite[Lemma 1.4]{Hi02}, which treats the $q$-Fock space for $q\in (-1,1)$ but the same proof works the case $q=-1$ (the $q$-Fock space is canonically isomorphic to $\cF_\anti$ when $q=-1$).

Since $J\ell(\xi)J = r(I\xi)$ for $\xi\in H$, the commutant of $ \Gamma_\anti(H_\R,H)$ is generated by 
	$$ W_r(I\, \widehat{\xi}) := JW(\widehat{\xi})J=r(I\, \widehat{\xi}) + r(I\,( \widehat{I\xi}))^* .$$
Write the C$^*$ and the von Neumann algebra generated by $W_r(I\, \widehat{\xi})$ for $\xi\in H$ as
	$$\mathrm{C}^*_{\anti,r}(H_\R,U):= J\mathrm{C}^*_{\anti}(H_\R,U)J ,\quad  \Gamma_{\anti,r}(H_\R,U):= J \Gamma_{\anti}(H_\R,U)J =  \Gamma _{\anti}(H_\R,U)'.$$

\subsection{Quasi-free states on CAR algebras}
\label{Quasi-free states on CAR algebras}

We keep the notation from the last subsection. We explain   $\mathrm{C}^*_\anti(H_\R,U)$ has a structure of the CAR algebra. We refer the reader to \cite[Section 6]{EK98} for all facts in this subsection. 

Assume first that $H_\R\subset H$ (with $H,I$) is given and $H$ has dimension $2d$. We say that a projection $P\in \B(H)$ is a \textit{basis projection} if $IPI=1-P$. We say that a linear map $B$ from $H$ to a unital C$^*$-algebra $C$ satisfies the \textit{self dual CAR} if 
	$$ B(\xi)^*=B(I\xi),\quad B(\xi)B(\eta)^*+B(\eta)^*B(\xi)=\langle \xi,\eta \rangle 1_C,\quad \text{for all }\xi,\eta\in H.$$
For any such $P$ and $B$, the restriction $B|_{PH}$ satisfies the CAR. Hence $\{B(\xi)\}_{\xi \in PH}$ generates the UHF algebra of type 2$^d$. More precisely, for any orthonormal basis $\{e_n\}_{n=1}^d$ in $PH$, we put $c_n:= B(e_n)$, $u_n:=1-2c_nc_n^*$, $v_n:=u_1\cdots u_n$ (here $v_0:=1$), and
	$$ \left[
\begin{array}{cc}
e^n_{11} &e^n_{12} \\
e^n_{21} &e^n_{22} \\
\end{array}
\right] := 
\left[
\begin{array}{cc}
c_n^*c_n &v_{n-1}c_n^* \\
c_n v_{n-1} &c_nc_n^* \\
\end{array}
\right] ,\quad \text{for all }1\leq n\leq d.$$
Then by the CAR equation, $\{e^n_{ij}\}_{ij}$ is a matrix unit for each $n$ and they are commuting for different $n$. Thus $B(PH)$ generates  the UHF algebra of type 2$^d$.

Next we consider $U\colon \R\act H_\R$ with $A$ as in the last subsection and define $\mathrm{C}^*_\anti(H_\R,U)$. Assume that there is a basis projection $P\in \B(H)$ such that it commutes with $U_t$ for all $t\in \R$. Then the map
	$$ B\colon H\to \B(\cF_\anti(H));\quad B(\xi):=\frac{1}{\sqrt{2}}\, (\ell(\widehat{\xi}) + \ell(\widehat{I\xi})^*) = \frac{1}{\sqrt{2}}\, W(\widehat{\xi})$$
satisfies the self dual CAR. In particular, $B(PH)$ generates the UHF algebra of type 2$^d$, which is by definition coincides with $\mathrm{C}^*_\anti(H_\R,U)$. In this way, we can identify $\mathrm{C}^*_\anti(H_\R,U)$ as the C$^*$-algebra generated by the CAR family $\{B(\xi)\}_{\xi \in PH}$. 
By putting $c(\xi):=B(P\xi)$ for $\xi\in PH$, the vacuum state $\varphi$ on $\mathrm{C}^*_\anti(H_\R,U)$ satisfies for all $\xi_1,\ldots,\xi_n,\eta_1,\ldots,\eta_m\in PH$,
	$$ \varphi(c(\eta_m)^*\cdots c(\eta_1)^* c(\xi_1)\cdots c(\xi_n))=\delta_{mn}\mathrm{det}[\langle R(A)\xi_i,\eta_j\rangle]_{ij} ,$$
where $R(A):=(1+A^{-1})^{-1} \in \B(PH)$. A (unique) state satisfying this condition is called the \textit{quasi-free state} for $R(A)$, denoted by $\omega_{R(A)}$. We get $\varphi = \omega_{R(A)}$.

	Let $V\colon H\to H$ be a unitary such that $VP=PV$, $VI=IV$, but possibly not commuting with $U_t$ for $t\in \R$. In this case, since $PH\ni \xi\mapsto c(V\xi)$ satisfies the CAR, by the uniqueness of the CAR algebra, the map
	$$ \alpha_V\colon \mathrm{C}^*_\anti(H_\R,U)\ni c(\xi)\mapsto c(V\xi)  \in \mathrm{C}^*_\anti(H_\R,U),\quad \text{for }\xi\in PH$$
is a well defined $\ast$-isomorphism. Then $\alpha_V$ extends to an isomorphism on $ \Gamma_\anti(H_\R,U)$ if and only if for $R_1:=R(A)$ and $R_2:=VR(A)V^*=R(VAV^*)$ in $\B(PH)$, $R_1^{1/2}-R_2^{1/2}$ and $(1-R_1)^{1/2}-(1-R_2)^{1/2}$ are Hilbert--Schmidt operators. This follows by the theory of quasi-free states.

\section{Boundary amenability on anti-symmetric Fock spaces}
\label{Boundary amenability on anti-symmetric Fock spaces}

In this section, we construct a boundary C$^*$-algebra acting on a given anti-symmetric Fock space. We prove some amenability result for the boundary.

\subsection{$\ell^\infty$-functions acting on anti-symmetric Fock spaces}\label{functions acting on anti-symmetric Fock spaces}

Let $X$ be a set and we consider $\ell^2(X)$. Define associated Fock spaces
\begin{align*}
	&\cF_\full:=\cF_\full(\ell^2(X)) =\C \Omega \oplus \bigoplus_{n \geq 1}\ell^2(X)^{\otimes n} ;\\
	&\cF_\anti:=\cF_\anti(\ell^2(X)) =\C \Omega \oplus \bigoplus_{n \geq 1}\ell^2(X)^{\wedge n}.
\end{align*}
Observe that each $\ell^2(X)^{\otimes n}$ in $\cF_{\rm full}$ has an identification
	$$\ell^2(X)^{\otimes n} = \ell^2(X^{n});\quad \delta_{x_1}\otimes \cdots \otimes \delta_{x_n} = \delta_{(x_1, \ldots , x_n)}.$$
So by using a large set 
	$$ \sfX := \{ \star \} \sqcup \bigsqcup_{n \geq 1} X^n,$$
where $\{\star\}$ is the singleton, we can identify 
	$$\ell^2 (\sfX) = \cF_{\rm full},\quad \text{with }\C \delta_\star = \C \Omega.$$
In this subsection, we frequently regard $\cF_\full$ (and hence $\cF_\anti$) as a function space.

For each $n\geq 1$, consider the natural action of $\frakS_n$ on $X^n$ by permutations:
	$$\sigma\cdot (x_1,\ldots,x_n):=(x_{\sigma(1)},\ldots,x_{\sigma(n)})\quad \text{for }\sigma\in \frakS_n,\ (x_1,\ldots,x_n)\in X^n.$$
 Put $Y_n:=X^n/\frakS_n$ and we denote by $[x]=[x_1,\ldots,x_n]$ the image of $x=(x_1,\ldots,x_n)\in X^n$ in $Y_n$. Define 
	$$Z_n:=\{[x_1,\ldots,x_n] \in Y_n\mid x_i\neq x_j \quad \text{for all }i, j\in \{1,\ldots,n\} \text{ with }i\neq j\}$$
and put
	$$\sfZ := \{\star\} \sqcup \bigsqcup_{n=1}^\infty Z_n,\quad \ell^\infty(\sfZ)=\ell^\infty(\{\star\})\oplus \bigoplus_{n\geq 1}\ell^\infty(Z_n).$$
The goal of this subsection is to prove the following proposition. We will construct our boundary via the embedding of this proposition.

\begin{prop}\label{prop-embedding}
	There is a unique injective $\ast$-homomorphism $\iota\colon \ell^\infty(\sfZ)\to \B(\cF_\anti)$ that satisfies the following condition: 
for any $n\in \N$, $f\in \ell^\infty(Z_n)$ and distinct $x_1,\ldots,x_n\in X$ (i.e.\ $x_i\neq x_j$ for all $i\neq j$), we have
	$$\iota(f)\, \delta_{x_1}\wedge \cdots \wedge \delta_{x_n} = f([x_1,\ldots,x_n])\, \delta_{x_1}\wedge \cdots \wedge \delta_{x_n}.$$
It holds that $\iota(\ell^\infty(\sfZ))\cap \K(\cF_\anti) = \iota(c_0(\sfZ))$.
\end{prop}

The permutation $\frakS_n\act X^n$ induces actions 
	$$\frakS_n\act \ell^p(X^n);\quad \sigma(f):=f\circ \sigma^{-1} \quad \text{for }f\in \ell^p(X^n),\ \sigma\in \frakS_n$$
where $p\in\{2, \infty\}$. For $f\in \ell^p(X^n)$, we say that $f$ is \textit{symmetric} if 
	$$\sigma (f)=f\quad \text{for all }\sigma\in \frakS_n.$$
We say that $f$ is \textit{anti-symmetric} if 
	$$\sigma (f)=\mathrm{sgn}(\sigma)f \quad \text{for all }\sigma\in \frakS_n.$$
Define two closed subspaces
\begin{align*}
	&\ell^\infty_\sym(X^n): = \{f\in \ell^\infty(X^n)\mid \text{symmetric}\};\\
	&\ell^2_\anti(X^n) := \{f\in \ell^2(X^n)\mid \text{anti-symmetric}\}.
\end{align*}

\begin{lem}\label{lem-embedding}
The following assertions hold true.
\begin{enumerate}
	\item For $f\in \ell^\infty_\sym(X^n)$ and $g\in \ell^2_\anti(X^n)$, we have $fg\in \ell^2_\anti(X^n)$. 

	\item Under the identification $ \ell^2(X^n)=\ell^2(X)^{\otimes n}$, the subspace $\ell^2_\anti(X^n)$ coincides with $\ell^2(X)^{\wedge n}$.

\end{enumerate}
\end{lem}
\begin{proof}
	1. This is trivial by definitions.

	2. Recall that $\ell^2(X)^{\wedge n}=\{\xi\in \ell^2(X)^{\otimes n}\mid \pi^{(n)}_\sigma\xi = \mathrm{sgn}(\sigma)\xi,\ \text{for all }\sigma\in \frakS_n\}$. Since $\pi^{(n)}_\sigma\xi$ coincides with $\xi\circ \sigma^{-1}$ as a function, $\ell^2(X)^{\wedge n}$ is exactly  $\ell^2_\anti(X^n)$.
\end{proof}

\begin{proof}[Proof of Proposition \ref{prop-embedding}]
	Since $Y_n=X^n/\frakS_n$, there is a canonical identification
	$$\ell^\infty_\sym(X^n) = \ell^\infty(X^n)^{\frakS_n} = \ell^\infty(Y_n).$$
This means that any $f\in \ell^\infty_\sym(X^n)$ can be regarded as a function on $Y_n$ by $f([x_1,\ldots,x_n]):=f(x_1,\ldots,x_n)$ for any $(x_1,\ldots,x_n)\in X^n$. By Lemma \ref{lem-embedding}, we have a $\ast$-homomorphism 
	$$\iota_n\colon \ell^\infty(Y_n)=\ell^\infty_\sym(X^n)\to \B(\ell^2_\anti(X^n))=\B(\ell^2(X)^{\wedge n}),$$
which satisfies
	$$ \iota_n(f)g = fg,\quad \text{for all }f\in \ell^\infty_\sym(X^n),\ g\in \ell^2_\anti(X^n) .$$
Observe that $\ell^\infty(Z_n)$ is a (non-unital) subalgebra in $\ell^\infty(Y_n)$ (since $Z_n\subset Y_n$). We restrict $\iota_n$ on $\ell^\infty(Z_n)$ and define $\iota:=\bigoplus_n\iota_n\colon \ell^\infty(\sfZ) \to \B(\cF_\anti)$. We will show that this is the desired one.

For distinct $x_1,\ldots,x_n\in X$, we regard $g:=\delta_{x_1}\wedge \cdots \wedge \delta_{x_n}\in \ell^2(X)^{\wedge n}$ as a function on $X^n$ by Lemma \ref{lem-embedding}.2, so that for any $y=(y_1,\ldots,y_n)\in X^n$,
\begin{align*}
	g(y) 
	&= \left(\frac{1}{n!}\sum_{\sigma\in \frakS_n}\mathrm{sgn}(\sigma)\, \delta_{x_{\sigma(1)}}\otimes  \cdots \otimes \delta_{x_{\sigma(n)}}\right)(y)\\
	&=\frac{1}{n!}\sum_{\sigma\in \frakS_n}\mathrm{sgn}(\sigma)\, \delta_{x_{\sigma(1)}}(y_1) \cdots \delta_{x_{\sigma(n)}}(y_n)
\end{align*}
For any $f\in \ell^\infty_\sym(X^n)=\ell^\infty(Y_n)$, it holds that
\begin{align*}
	(\iota_n(f)g) (y)
	&=f([y_1,\ldots,y_n])\, \frac{1}{n!}\sum_{\sigma\in \frakS_n}\mathrm{sgn}(\sigma)\, \delta_{x_{\sigma(1)}}(y_1) \cdots \delta_{x_{\sigma(n)}}(y_n).
\end{align*}
Observe that if there is no $\tau\in \frakS_n$ such that $(x_{\tau(1)},\ldots,x_{\tau(n)})=(y_1,\ldots,y_n)$, then the last term is zero. 
Assume that there is $\tau\in \frakS_n$ such that $(x_{\tau(1)},\ldots,x_{\tau(n)})=(y_1,\ldots,y_n)$. Then the last term is
\begin{align*}
	&=f([y_1,\ldots,y_n])\, \frac{1}{n!}\, \mathrm{sgn}(\tau)\, \delta_{x_{\tau(1)}}(y_1) \cdots \delta_{x_{\tau(n)}}(y_n)\\
	&=f([x_{\tau(1)},\ldots,x_{\tau(n)}])\, \frac{1}{n!}\, \mathrm{sgn}(\tau)\\
	&=f([x_{1},\ldots,x_{n}])\, \frac{1}{n!}\, \mathrm{sgn}(\tau).
\end{align*}
Combining these observations together, we conclude 
	$$ (\iota_n(f)g) (y) = f([x_{1},\ldots,x_{n}]) g(y) ,\quad \text{for  all }y\in X^n.$$
If $f\in \ell^\infty(Z_n)\subset \ell^\infty(Y_n)$, this is the desired condition for $\iota$.

We next show that $\iota$ is unique and injective on $\ell^\infty(\sfZ)$. The uniqueness is obvious since $\delta_{x_1}\wedge\cdots \wedge \delta_{x_n}$ for $n$ and distinct $x_1,\ldots,x_n\in X$ span a dense subspace of $\cF_\anti$. To see the injectivity, take any nonzero $f\in \ell^\infty(\sfZ)$ and pick up $n\in \N$ and $z\in Z_n$ such that $f(z)\neq 0$. Then writing $z=[x_1,\ldots,x_n]$, we have
	$$ \iota(f)\, \delta_{x_1}\wedge \cdots \wedge \delta_{x_n} = f([x_{1},\ldots,x_{n}]) \, \delta_{x_1}\wedge \cdots \wedge \delta_{x_n}\neq 0$$
This shows $\iota(f)\neq 0$ and hence $\iota$ is injective on $\ell^\infty(\sfZ)$.

We prove the last part of the statement. If $f\in c_0(\sfZ)$ has finite support, then by the equation in the statement which we have already proved, the range of $\iota(f)$ is finite dimensional. This observation shows that $\iota(c_0(\sfZ))\subset \iota(\ell^\infty(\sfZ))\cap \K(\cF_\anti) $. To see the converse, take any $\iota(f)\in \iota(\ell^\infty(\sfZ))\cap \K(\cF_\anti)$. Take any net $(z_\lambda)_{\lambda}$ in $\sfZ$ such that $z_\lambda \to \infty$. Take any representative $(x_1^\lambda,\ldots,x^\lambda_{n_\lambda})\in X^{n_\lambda}$ such that $z_\lambda = [x_1^\lambda,\ldots,x^\lambda_{n_\lambda}]$. 
Put $\xi_\lambda := \delta_{x_1^\lambda}\wedge \cdots \wedge \delta_{x^\lambda_{n_\lambda}}$ and $\eta_\lambda:=\xi_\lambda/\|\xi_\lambda\|$. Then observe that $z_\lambda \to \infty$ implies $\eta_\lambda \to 0$ weakly. Since $\iota(f)$ is compact, $\iota(f)\eta_\lambda = f(z_\lambda)\eta_\lambda \to 0$ in norm. This means $f(z_\lambda)\to 0$ and $f$ is in $c_0(\sfZ)$.
\end{proof}

\subsection{Definition of boundary}
\label{Definition of boundary}

Let $X$ be a set and consider $\cF_\anti:=\cF_\anti(\ell^2(X))$. We use the embedding $\ell^\infty(\sfZ)\subset \B(\cF_\anti)$ in Proposition \ref{prop-embedding} (we omit $\iota$ for simplicity). Define
	$$ C(\overline{\sfZ}):=\{ f\in \ell^\infty(\sfZ)\mid [f,a]\in \K(\cF_\anti)\ \text{for all }a\in \cC_{\ell,r} \} ,$$
where $\cC_{\ell,r}$ is the C$^*$-algebra generated by all left and right creation operators. Note that it is generated by $\ell(x):=\ell(\delta_x)$ and $r(x):=r(\delta_x)$ for all $x\in X$. 
Since $c_0(\sfZ)\subset \K(\cF_\anti)$, $C(\overline{\sfZ})$ contains $c_0(\sfZ)$, hence we can define a quotient C$^*$-algebra
	$$C(\partial \sfZ):=C(\overline{\sfZ})/c_0(\sfZ).$$
This should be regarded as a \textit{boundary} with respect to left and right creation operators. We will observe that any group action on $X$ induces a natural action on $C(\partial \sfZ)$, which we regard a \textit{boundary action}.

Let $ G \act X$ be an action of a discrete group $ G $. Consider the associated unitary representation
	$$\pi\colon G  \act \ell^2(X);\quad \pi_g\delta_x=\delta_{g\cdot x}$$
and the associated one on $\cF_\anti$ by 
	$$\pi_g^\anti =1_\Omega \oplus \bigoplus_{n \geq 1} \pi_g^{\wedge n},\quad g\in  G  .$$
We also consider a natural diagonal action $ G  \act \sfZ$ given by 
	$$g\cdot [x_1,\ldots,x_n]=[g\cdot x_1,\ldots,g\cdot x_n]\quad \text{for }g\in G, \ n\in \N , \ [x_1,\ldots,x_n]\in Z_n\subset \sfZ.$$
This induces an action $ G  \act \ell^\infty(\sfZ)$ by $g\cdot f:=f\circ g^{-1}$ for $f\in \ell^\infty(\sfZ)$ and $g\in  G $.

\begin{lem}\label{lem-action-boundary1}
	The following conditions hold true.
\begin{enumerate}
	\item For any $g\in  G $ and $f\in \ell^\infty(\sfZ)$, we have  $\pi_g^\anti f \pi_g^{\anti *} = g\cdot f$.

	\item For $f\in \ell^\infty(\sfZ)$, $f$ is contained in $C(\overline{\sfZ})$ if and only if $[f,a]\in \K(\cF_\anti)$ for all $a=\ell(x),\ell(x)^*,r(x),r(x)^*$ and $x\in X$.

\end{enumerate}
\end{lem}
\begin{proof}
	1. For any distinct $x_1,\ldots,x_n\in X$, 
\begin{align*}
	f\, \pi_g^\anti (\delta_{x_1}\wedge \cdots \wedge \delta_{x_n} )
	&=f \, (\delta_{g\cdot x_1}\wedge \cdots \wedge  \delta_{g\cdot x_n} )\\
	&=f([g\cdot x_1,\ldots,g\cdot x_n]) (\delta_{g\cdot x_1}\wedge \cdots \wedge  \delta_{g\cdot x_n} )\\
	&=(g^{-1}\cdot f)([x_1,\ldots,x_n]) \pi_g^\anti(\delta_{x_1}\wedge \cdots \wedge  \delta_{ x_n} )\\
	&=\pi_g^\anti(g^{-1}\cdot f)\, (\delta_{x_1}\wedge \cdots \wedge  \delta_{ x_n} ).
\end{align*}
This means $\pi_g^{\anti *}f\pi_g^\anti = g^{-1}\cdot f$.

	2. Assume that $[f,a]\in\K(\cF_\anti)$ for $a=\ell(x),\ell(x)^*,r(x),r(x)^*$ for $x\in X$. Then since $\K(\cF_\anti)$ is an ideal in $\B(\cF_\anti)$, the same holds if $a$ is a product of such elements. Then it is easy to see that $[f,a]\in\K(\cF_\anti)$ for all $a\in \cC_{\ell,r}$.
\end{proof}

\begin{lem}\label{lem-action-boundary2}
	The action $ G  \act \ell^\infty(\sfZ)$ preserves $C(\overline{\sfZ})$, hence it induces an action $ G  \act C(\partial \sfZ)$.
\end{lem}
\begin{proof}
For any $x\in X$, $g\in  G $, $f\in C(\overline{\sfZ})$,
\begin{align*}
	[\ell(x),g\cdot f]
	=[\ell(x),\pi_g^\anti f \pi_g^{\anti *}]
	&= \pi_g^\anti[\pi_g^{\anti *}\ell(x)\pi_g^\anti, f ]\pi_g^{\anti *}\\
	&= \pi_g[\ell(g^{-1}\cdot x), f] \pi_g^*\in \K(\cF_\anti). 
\end{align*}
Similar computations yield $[a,g\cdot f]\in \K(\cF_\anti)$, where $a=\ell(x)^*,r(x)$ and $r(x)^*$. By Lemma \ref{lem-action-boundary1}, we get $g\cdot f\in C(\overline{\sfZ})$.
\end{proof}

\subsection{Boundary amenability}

We prove a boundary amenability result. More precisely we prove the following.

\begin{thm}\label{thm-boundary-amenability}
	Let $ G $ be a countable discrete group acting on a countable set $X$.  Assume that $ G $ is exact and the action has finite stabilizers and finitely many orbits. Then the action $ G  \act C(\partial \sfZ)$ given in Lemma \ref{lem-action-boundary2} is amenable. 
\end{thm}

Our proof is inspired by Ozawa's work on Bernoulli actions \cite[Proposition 4.4]{Oz04}\cite[Corollary 15.3.9]{BO08}. We use a similar idea but in a different context.

We need several preparations. Recall that any countable group $ G $ with any finite subgroup $\Lambda$ admits a proper length function $| \cdot | \colon  G /\Lambda \to \R_{\geq 0}$ satisfying
\begin{itemize}
	\item[$\rm (i)$] $|g\Lambda| =0 \Leftrightarrow g \in \Lambda$; 
	\item[$\rm (ii)$] $|gh\Lambda|\leq |g\Lambda|+  |h\Lambda|$;
	\item[$\rm (iii)$] (symmetric) $|g^{-1}\Lambda| = |g\Lambda|$;
	\item[$\rm (iv)$] (proper) $\{g\Lambda \mid |g\Lambda| \leq R \}$ is finite for all $R> 0$.
\end{itemize}
(See \cite[Proposition 5.5.2]{BO08}.) We use the following functions.

\begin{lem}\label{length function lemma}
	Assume that $ G $ is countable and $ G  \act X$ has finite stabilizers and finitely many orbits. Then there is a function $| \cdot |_X\colon  G /\Lambda \to \R_{\geq 0}$, a finite subgroup $\Lambda \leq  G $, and a proper symmetric length function $|\cdot|_{ G /\Lambda}\colon  G /\Lambda \to \R_{\geq 0}$ such that 
\begin{itemize}
	\item[$\rm (i)$] $|g\cdot x|_X\leq |g\Lambda|_{ G /\Lambda} + |x|_X$;
	\item[$\rm (ii)$] $\{x\in X \mid |x|_X \leq R \}$ is finite for all $R> 0$.
\end{itemize}
\end{lem}
\begin{proof}
By assumption, there exist finite subgroups $\Lambda_1,\ldots,\Lambda_n\leq  G $ such that $X \simeq \bigsqcup_{i=1}^n  G  /\Lambda_i$. Put $\Lambda:=\bigcap_i\Lambda_i \leq  G $. 
For each $i$, take any length function $|\cdot|_i$ for $ G /\Lambda_i$. Define $|\cdot|_X$ by $|x|_X:=|x|_{i}$ for $x\in  G /\Lambda_i\subset X$. Observe that $|g\Lambda|_{ G /\Lambda}:=\max_{i=1}^n|g\Lambda_i|_i$ defines a length function on $ G /\Lambda$. Then it is easy to check the conditions.
\end{proof}

Fix $ G  \act X$ as in the statement of Theorem \ref{thm-boundary-amenability} and we keep the notation from previous subsections. We fix $|\cdot|_X$ and $|\cdot|_{ G }(:=|\cdot|_{ G /\Lambda})$ in Lemma \ref{length function lemma}. 

Define two well defined functions on $\sfZ$: for any $z =[x_1,\dots,x_n] \in Z_n\subset \sfZ$,
	$$|z|_0 :=n,\quad |z|_1 := \sum_{i=1}^n|x_i|_X,$$
and $|\star|_0=|\star|_1 = 0$. We use them with the following elementary lemma.

\begin{lem}\label{length-convergence-lem}
	Let $(z_\lambda)_\lambda$ be any net in $\mathsf{X}_\anti$. Then we have
	$$ z_\lambda \to \infty \quad \Leftrightarrow \quad \lim_\lambda |z_\lambda|_0 + |z_\lambda|_1 = \infty.$$
\end{lem}
\begin{proof}
	$(\Rightarrow)$ Fix $R>0$ and we show that the set of all $z\in \sfZ$ satisfying $|z|_0+|z|_1\leq R$ is finite. Fix such $z$ and write $z=[x_1,\ldots,x_n]$. Note that $n=|z|_0\leq R$. Since $|x_i|\leq |z|_1\leq R$, $x_i$ is contained in $B_R(X):=\{x\in X\mid |x|_X\leq R\}$, which is a finite set. Thus any representative $(x_1,\ldots,x_n)$ of $z$ is contained in a finite set $B_R(X)^n$, so that such $z$ is contained in a finite set which depends only on $R$. 

	$(\Leftarrow)$ If $z_\lambda \not \to \infty$,  then a subnet $\{z_{\lambda'}\}_{\lambda'}$ is contained in a finite subset in $\sfZ$, which can not satisfy $\lim_{\lambda'} |z_{\lambda'}|_0 + |z_{\lambda'}|_1 = \infty$. 
\end{proof}

Define a map $\omega \colon \sfZ \to \ell^1(X)^+$ by 
\[
\omega (z) := \sum_{i=1}^n (|z|_0 + |x_{i}|_X) \delta_{x_i} \quad \text{for} \quad z=[x_1,\ldots,x_n] \in Z_n\subset \sfZ,
\]
where $\omega (\star)$ is any positive value. Note that 
	$$ \|\omega(z) \|_1 = \sum_{i=1}^n |z|_0 + |x_{i}|_X = |z|_0^2 + |z|_1.$$
Up to normalizing, we define 
\[
	\mu\colon \sfZ \to \Prob(X);\quad \mu (z) := \frac{\omega(z)}{\|\omega(z)\|_1}.
\]
We can induce the following ucp map
\[
	\mu^* \colon \ell^\infty (X) \to \ell^\infty (\sfZ);\quad f\mapsto [\sfZ \ni z \mapsto \langle f, \mu(z)\rangle ],
\]
where $ \langle f, \mu(z)\rangle $ is the paring by $\ell^1(X)^*=\ell^\infty(X)$. The next lemma is the main step to prove Theorem \ref{thm-boundary-amenability}. 

\begin{lem}\label{lem-boundary-amenability}
	The following conditions hold true.
\begin{enumerate}
	\item For any $g \in  G $ and $\varphi \in \ell^\infty (X)$, we have
	$$\mu^* (g \cdot \varphi ) - g\cdot \mu^*( \varphi ) \in c_0(\sfZ).$$

	\item For any $x\in X$ and $\varphi\in \ell^\infty(X)$, we have
\begin{align*}
	&[\mu^*(\varphi),\ell(x)],\ [\mu^*(\varphi),r(x)]\in \K(\cF_\anti).
\end{align*}

\end{enumerate}
\end{lem}
\begin{proof}
	1. Fix $g\in  G  $ and $\varphi \in \ell^\infty(X)$. For any $z\in \sfZ$, 
\begin{align*}
	 \mu^*( g\cdot \varphi )(z) -  g\cdot \mu^*( \varphi )(z) 
	= \langle g\cdot \varphi,\mu(z) \rangle - \langle \varphi,\mu(g^{-1}\cdot z) \rangle 
	= \langle \varphi,g^{-1} \cdot \mu(z) -\mu(g^{-1}\cdot z) \rangle .
\end{align*}
We have to show that this converges to $0$, as $z\to \infty$. For this, we prove
	$$\lim_{\sfZ\ni z\to \infty}\|\mu (g \cdot z) - g\cdot \mu(z) \|_1=0. $$

\begin{claim}
We have $ \|g\cdot \omega(z)-\omega(g\cdot z)\|_1\leq |z|_0|g|_ G  $ for any $z\in \sfZ$.
\end{claim}
\begin{proof}
	Write $z=[a_1,\ldots,a_n]\in Z_n$ and observe
\begin{align*}
	g\cdot \omega(z)
	=\sum_{i=1}^n (n + |a_{i}|_X) \delta_{g\cdot a_i} , \quad
	\omega(g\cdot z)
	=\sum_{i=1}^n (n + |g\cdot a_i|_X) \delta_{g\cdot a_i}.
\end{align*}
It follows that
\begin{align*}
	\|g\cdot \omega(z)-\omega(g\cdot z)\|_1
	&=\left\|\sum_{i=1}^n  (|a_{i}|_X - |g\cdot a_i|_X|) \delta_{g\cdot a_i}\right\|_1\\
	&\leq \sum_{i=1}^n  ||a_{i}|_X - |g\cdot a_i|_X||
	\leq \sum_{i=1}^n  |g|_ G  = n|g|_ G .
\end{align*}
\end{proof}
Now we compute that
\begin{align*}
	\|g\cdot\mu(z) -\mu(g\cdot z)\|_1 
	&= \left\|\frac{g\cdot\omega(z)}{\|\omega(z)\|_1} - \frac{\omega(g\cdot z)}{\|\omega(g\cdot z)\|_1}\right\|_1 \\
	&\leq \left\|\frac{g\cdot\omega(z)}{\|\omega(z)\|_1} - \frac{\omega(g\cdot z)}{\|\omega(z)\|_1}\right\|_1 +\left\|\frac{\omega(g\cdot z)}{\|\omega(z)\|_1} - \frac{\omega(g\cdot z)}{\|\omega(g\cdot z)\|_1}\right\|_1\\
	&\leq \frac{1}{\|\omega(z)\|_1}\left\|g\cdot\omega(z) - \omega(g\cdot z)\right\|_1 + \left|\frac{1}{\|\omega(z)\|_1} - \frac{1}{\|\omega(g\cdot z)\|_1}\right|\|\omega(g\cdot z)\|_1\\
	&\leq \frac{2}{\|\omega(z)\|_1}\left\|g\cdot\omega(z) - \omega(g\cdot z)\right\|_1 .
\end{align*}
Combined with the clam, we get
\begin{align*}
	\|g\cdot\mu(z) -\mu(g\cdot z)\|_1 
	\leq 2|g|_ G  \frac{|z|_0}{|z|_0^2+|z|_1}.
\end{align*}
If we consider $z=z_\lambda \to \infty$, then by Lemma \ref{length-convergence-lem}, this means $|z_\lambda|_0+|z_\lambda|_1\to \infty$. It is straightforward to see that the last term converges to $0$, as desired.

	2. We first see an elementary claim.
\begin{claim}
	Let $f\in \ell^\infty(\sfZ)$ and $x\in X$.
\begin{enumerate}
	\item We have $ f\ell(x)=\ell(x)f([x,\,\cdot \,])$  and $fr(x)=r(x)f([\,\cdot \,, x])$.
	\item Set $ Z_x:=\{ [x_1,\ldots,x_n]\in \sfZ \mid n\in \N,\  x_i\neq x \text{ for all }i\} $. Then we have 
	$$ \ell(x)f = \ell(x)f1_{Z_x},\quad r(x)f = r(x)f1_{Z_x}.$$
\end{enumerate}
\end{claim}
\begin{proof}
	1. It is straightforward by the equation in Proposition \ref{prop-embedding}.

	2. Since $f=f1_{Z_x}+f1_{Z_x^c}$, we have only to see $\ell(x)f1_{Z_x^c}=0 =r(x)f1_{Z_x^c}$. For any $n\in \N$ and distinct $x_1,\ldots,x_n\in X$, compute that
\begin{align*}
	\ell(x)f1_{Z_x^c}\, \delta_{x_1}\wedge \cdots \wedge \delta_{x_n}
	&= f([x_1,\ldots,x_n])1_{Z_x^c}([x_1,\ldots,x_n]) \ell(x)\, \delta_{x_1}\wedge \cdots \wedge \delta_{x_n}\\
	&=\sqrt{n+1}\,  f([x_1,\ldots,x_n])1_{Z_x^c}([x_1,\ldots,x_n])\,  \delta_{x}\wedge \delta_{x_1}\wedge \cdots \wedge \delta_{x_n}.
\end{align*}
If $x=x_i$ for some $i$, then $\delta_{x}\wedge \delta_{x_1}\wedge \cdots \wedge \delta_{x_n}=0$. If $x\neq x_i$ for all $i$, then $[x_1,\ldots,x_n]\in Z_x$ and $1_{Z_x^c}([x_1,\ldots,x_n])=0$. So in any case, the above term is $0$. This means $\ell(x)f1_{Z_x^c}=0$. A similar computation works for $r(x)f1_{Z_x^c}$.
\end{proof}

We fix $f:=\mu^*(\varphi)$ and $x\in X$ in the statement. By the claim,
\begin{align*}
	 [\ell(x) , f] = \ell(x)(f-f([x,\,\cdot \,]))1_{Z_x},\quad 
	 [r(x) , f] = r(x)(f-f([\,\cdot \, ,x])1_{Z_x}.
\end{align*}
Since $f([x,\,\cdot \,]) = f([\,\cdot \, , x])$, we have only to prove
	$$(f-f([x,\,\cdot \,]))1_{Z_x}\in c_0(\sfZ).$$
Let $z=[a_1,\ldots,a_n]\in Z_x$ and observe that
\begin{align*}
	f(z) -  f([x,a_1,\ldots,a_n]) 
	= \langle \varphi,\mu(z) \rangle - \langle \varphi,\mu([x,a_1,\ldots,a_n]) \rangle 
	= \langle \varphi,[\mu(z) -\mu([x,a_1,\ldots,a_n]) ]\rangle .
\end{align*}
We have to show that it converges to $0$ as $Z_x\ni z\to \infty$. We indeed prove
	$$\lim_{Z_x\ni z=[a_1,\ldots,a_n]\to \infty}\|\mu(z) -\mu([x,a_1,\ldots,a_n]) \|_1=0 .$$

\begin{claim}
	For any $z=[a_1,\ldots,a_n]\in Z_x$, it holds
	$$\|\omega(z) - \omega([x,a_1,\ldots,a_n])\|_1\leq 2|z|_0 + 1 + |x|_X.$$ 
\end{claim}
\begin{proof}
Recall $\omega(z)=\sum_{i=1}^{n} (n + |a_i|_X) \delta_{a_i}$ and observe
\begin{align*}
	\omega([x,a_1,\ldots,a_n])
	&= (n+1+|x|_X)\delta_x + \sum_{i=1}^{n} (n+1 + |a_i|_X) \delta_{a_i}.
\end{align*}
It follows that
\begin{align*}
	\|\omega(z)-\omega([x,a_1,\ldots,a_n])\|_1
	&= \left\|(n+1 + |x|_X) \delta_{x} + \sum_{i=1}^{n}  \delta_{a_i}\right\|_1\\
	&\leq n+1 + |x|_X + n
	= 2|z|_0+1 + |x|_X .
\end{align*}
\end{proof}
Now a computation similar to the proof of item 1 shows that (with $w:=[x,a_1,\ldots,a_n]$)
\begin{align*}
	\|\mu(z) -\mu(w)\|_1 
	&\leq \frac{2}{\|\omega(z)\|_1}\left\|\omega(z) - \omega(w)\right\|_1 \\
	&\leq \frac{2}{\|\omega(z)\|_1}(2|z|_0 + 1 + |x|_X) 
	= \frac{2}{|z|_0^2 + |z|_1}(2|z|_0 + 1 + |x|_X).
\end{align*}
It is straightforward to see that it converges to $0$ as $Z_x\ni z\to \infty$.
\end{proof}

\begin{proof}[Proof of Theorem \ref{thm-boundary-amenability}]
	Let $Q\colon \ell^\infty (\sfZ)\to \ell^\infty (\sfZ)/c_0(\sfZ)$ be the quotient map. Observe that, by taking the involution, the conclusion of item 2 in Lemma \ref{lem-boundary-amenability} implies $[\mu^*(\varphi),a]\in \K(\cF_\anti)$ for $\varphi\in \ell^\infty(X)$, $a=\ell(x)^*,r(x)^*$ and $x\in X$. Hence by item 2 in Lemma \ref{lem-action-boundary1}, $\mu^*(\ell^\infty(X))$ is contained in $C(\overline{\sfZ})$. 
We get a ucp map
	$$Q\circ\mu^* \colon \ell^\infty (X) \to C(\overline{\sfZ})/c_0(\sfZ)=C(\partial \sfZ).$$
Item 1 in Lemma \ref{lem-boundary-amenability} shows that $Q\circ\mu^*$ is $ G $-equivariant. Now since $ G $ is exact and $ G  \act X$ has finite stabilizers and finitely many orbits, the action $ G  \act \ell^\infty(X)$ is amenable. Since $Q\circ\mu^*$ is a $ G $-equivariant ucp from $\ell^\infty (X)$ to $C(\partial \sfZ)$, it follows that $ G  \act C(\partial \sfZ)$ is amenable.
\end{proof}

\section{Rigidity of non-singular Bernoulli actions}\label{Rigidity of non-singular Bernoulli actions}

In this section, we prove the main theorem. As we have already explained, we use anti-symmetric Fock spaces and boundary amenability obtained in the last section.

\subsection{Bernoulli actions via the CAR construction}\label{Non-singular Bernoulli actions via the CAR construction}

	We explain how to use the CAR construction to study nonsingular Bernoulli actions. For this we first recall the structure of $ \Gamma _\anti(H_\R,U)$ when $U\colon \R\act H_\R$ is almost periodic.

Let $X_0$ be any countable set and take any positive values $\mu_x>0$ for $x\in X_0$. Define a (complex) Hilbert space $H$ and a positive self adjoint operator $A$ on $H$ by
	$$H := \bigoplus_{x\in X_0} \C\eta_x\oplus \C I\eta_x,\quad A:=\bigoplus_{x\in X_0} \mu_x\oplus \mu_x^{-1},$$
where $\eta_x$ and $I\eta_x$ are symbols of unit vectors. Define a strongly continuous unitary representation $U\colon \R \act H$ by $U_t=A^{\ri t}$, that is,
	$$U_t\eta_i=\mu_i^{\ri t}\eta_i \quad \text{and}\quad U_tI\eta_i=\mu_i^{-\ri t}I\eta_i,\quad \text{for all }i\in X_0, \ t\in \R.$$
Define a conjugate linear isometry $I\colon H\to H$ by $\eta_i\mapsto I\eta_i$ and $I\eta_i\mapsto \eta_i$ for all $i\in X_0$. We consider the real structure $H_\R:=\{\xi\in H\mid I\xi = \xi\} \subset H$, and observe that $U$ acts on $H_\R$ since $I$ and $U_t$ commute for all $t\in \R$. In this way, for any given $\{\mu_x\}_{x\in X_0}$, we can construct an almost periodic representation $U\colon \R\act H_\R$ and associated operator algebras
	$$\mathrm{C}^*_\anti( H_\R,U )\subset  \Gamma_\anti(H_\R,U),\quad \text{with the vacuum state }\varphi.$$
We call $(H_\R,U)$ the \textit{almost periodic representation with eigenvalues $\{\mu_x\}_{x\in X_0}$}.

\subsection*{Identification $H=\ell^2(X)$}

Let $(H_\R,U)$ the almost periodic representation with eigenvalues $\{\mu_x\}_{x\in X_0}$ as constructed. 
Put $X:=X_0\sqcup IX_0$, where $IX_0:=\{Ix\mid x\in X_0\}$ with symbol $I$, and we identify $H=\ell^2(X)$ by 
	$$\eta_x = \delta_{x}\quad \text{and}\quad I\eta_x=\delta_{Ix}\quad \text{for }x\in X_0.$$ 
If we consider an order $2$ bijection $I\colon X\to X$ by $x\mapsto Ix$ and $Ix\mapsto x$ for $x\in X_0$, then the induced conjugate linear isometry $I\colon \ell^2(X)\to \ell^2(X)$ given by $If = \overline{f\circ I}$, coincides with the involution $I$ on $H$ above. 
It holds that $I\delta_x = \delta_{Ix}$ for all $x\in X$. We note that the real part $H_\R$ coincides with
	$$ \{f + \overline{f\circ I} \mid f\in \ell^2(X)\}$$
which is different from $\ell^2_\R(X)$.

\subsection*{Araki--Woods factors}

Let $(H_\R,U)$ the almost periodic representation with eigenvalues $\{\mu_x\}_{x\in X_0}$ and identify $H=\ell^2(X)$ for $X=X_0\sqcup IX_0$. Let $P\colon \ell^2(X)\to\ell^2(X_0)$ be the orthogonal projection. Then $P$ is a basis projection for the involution $I$, namely, $IPI=1-P$, and it commutes with $U_t$ for $t\in \R$. Hence, if we put
	$$c_x:=\frac{1}{\sqrt{2}}W(\widehat{\delta}_x)=\frac{1}{\sqrt{2}}(\ell(\widehat{\delta}_x)+\ell(\widehat{\delta}_{Ix})^*),\quad \text{for all }x\in X ,$$
then $\{c_x\}_{x\in X}$ gives a self dual CAR family. The restriction $\{c_x\}_{x\in X_0}$ to $X_0$ is a CAR family which generates $\mathrm{C}^*_\anti(H_\R,U)$ as a C$^*$-algebra. In particular, we have an UHF isomorphism $\mathrm{C}^*_\anti(H_\R,U)\simeq \bigotimes_{x\in X_0}\M_2$ as explained in Subsection \ref{Quasi-free states on CAR algebras}. Under this isomorphism, the vacuum state $\varphi$ on $\mathrm{C}^*_\anti(H_\R,U)$, which is the unique quasi-free state $\omega_{R(A)}$ for $R(A)\in \B(PH)$, coincides with the product state $\psi=\otimes_{x\in X_0}\psi_x$ on $\bigotimes_{x\in X_0}\M_2$ given by  for all $x\in X_0$,
	$$ \psi_x= \mathrm{tr}\left(\left[
\begin{array}{cc}
p_x&0  \\
0& q_x\\
\end{array}
\right]   \, \cdot \, \right),\quad \text{where }	p_x= \frac{1}{1+\mu_{x}^{-1}}\quad \text{and}\quad
	q_x=\frac{\mu_{x}^{-1}}{1+\mu_{x}^{-1}}.
$$
This means that the isomorphism extends to a state preserving isomorphism 
	$$( \Gamma_\anti(H_\R,U),\varphi) \simeq \overline{\bigotimes}_{x\in X_0}(\M_2,\psi_x). $$
In this way, we can identify $ \Gamma_\anti(H_\R,U)$ with an Araki--Woods factor. If we consider a product measure space
	$$(\Omega_\Ber,\mu) := \prod_{x\in X_0} (\{0,1\},p_x\delta_0 + q_x \delta_1 ),$$
then $L^\infty(\Omega_\Ber,\mu)$ is (diagonally) contained in the Araki--Woods factor. Under the isomorphism, this coincides with the embedding
	$$ \mathrm{W}^*\{ c_x^*c_x,c_xc_x^*\mid x\in X_0 \} \subset  \Gamma_\anti(H_\R,U).$$
In this way, for any product measure space $(\Omega,\mu)$ with parameter $p_x,q_x$ with $p_x+q_x=1$ for $x\in X_0$, $L^\infty(\Omega,\mu)$ is contained in some $( \Gamma_\anti(H_\R,U),\varphi)$ with $\varphi$-preserving conditional expectation.

\subsection*{Associated group actions}

We keep the notation and let $ G \act X_0$ be an action of a countable group $ G $. We extend it to $X=X_0\sqcup IX_0$ by $g\cdot (Ix):=I(g\cdot x)$ for $x\in X_0$. We denote by $\pi\colon  G  \act \ell^2(X)$ the associated unitary representation. Observe that
	$$ \pi_g \eta_x = \eta_{g\cdot x}\quad \text{and}\quad \pi_g I\eta_x = I\eta_{g\cdot x}\quad \text{for all }x\in X_0.$$
Since $\pi_g$ commutes with $I$ and $P$ for $g\in  G $, $\pi$ induces an action $\alpha\colon  G  \act \mathrm{C}^*_\anti(H_\R,U)$  by 
	$$\alpha_g\colon \mathrm{C}^*_\anti(H_\R,U) \ni c_x \mapsto  c_{g\cdot x}\in \mathrm{C}^*_\anti(H_\R,U),\quad \text{for all }x\in X_0.$$
As explained in Subsection \ref{Quasi-free states on CAR algebras}, $\alpha_g$ extends to $ \Gamma _\anti(H_\R,U)$ if and only if, for $R:=R(A)$ and $R_g:=R(\pi_g A\pi_g^{-1})$, $R^{1/2} - R_g^{1/2}$ and $(1-R)^{1/2}-(1-R_g)^{1/2}$ are Hilbert--Schmidt operators for all $g\in  G $. This turns out to be equivalent to Kakutani's condition
 for $p_x,q_x$, that is,
	$$\sum_{x\in X_0}\left(\sqrt{p_x} - \sqrt{p_{g\cdot x}}\right)^2 +\sum_{x\in X_0}\left(\sqrt{q_x} - \sqrt{q_{g\cdot x}}\right)^2<\infty,\quad \text{for all }g\in  G .$$
By the explicit form of $\alpha$, the action restricts to the nonsingular Bernoulli action on $\mathrm{W}^*\{c_x^*c_x,c_xc_x^*\mid x\in X_0\}\simeq L^\infty(\Omega_\Ber,\mu)$. It is worth mentioning that the action is \textit{not} the Bernoulli action on the Araki--Woods factor. 

\subsection*{Conclusion of this subsection}

Summarizing above observations, we obtain the following lemma. By this lemma, to prove our main theorem, we have only to show study actions on $ \Gamma_\anti(H_\R,U)$. 

\begin{lem}\label{lem-Bernoulli-embedding}
	Let $ G  \act X_0$ be any action of a countable discrete group on a countable set. Consider any product measure space 
	$$(\Omega_\Ber,\mu):=\prod_{x\in X_0} (\{0,1\},p_x\delta_0 + q_x \delta_1).$$
Assume that $p_x,q_x$ for $x\in X_0$ satisfies Kakutani's condition for $ G \act X_0$, so that the nonsingular (generalized) Bernoulli action is defined. 

For $x\in X_0$, take a unique $\mu_x>0$ such that $p_x=\frac{1}{1+\mu_x^{-1}}$, $q_x=\frac{\mu_x^{-1}}{1+\mu_x^{-1}}$ and consider the almost periodic representation $U\colon \R\act H_\R$ with eigenvalues $\{\mu_x\}_{x\in X_0}$. Then 
\begin{itemize}
	\item the action $ G  \act X_0$ induces $\alpha\in \Aut( \Gamma_\anti(H_\R,U))$ by $\alpha_g(W(\widehat{\delta}_x)) = W(\widehat{\delta}_{g\cdot x})$ for $g\in  G $ and $x\in X_0$;
	\item there is an inclusion $L^\infty(\Omega_\Ber,\mu)\subset  \Gamma _\anti(H_\R,U)$ with expectation $E$ such that $\mu\circ E$ coincides with the vacuum state and $\alpha$ restricts to the nonsingular Bernoulli action on $L^\infty(\Omega_\Ber,\mu)$.
\end{itemize}
In this case, $ G  \act (\Omega_\Ber,\mu)$ is solid if $\alpha$ is solid.
\end{lem}
\begin{proof}
	As we have already explained, we can start from $(\Omega_\Ber,\mu)$ and then construct the associated $U\colon \R\act H_\R$. For the definition of solidity for general actions, see Subsection \ref{Definition and characterization}. By this definition and Theorem \ref{solid-thm}, the last statement is trivial.
\end{proof}

Throughout this section, we always use objects and notations above. Since they are complicated, for reader's convenience, we summarize our notations below.

\begin{notation}\upshape\label{notation}
	Let $ G  \act X_0$ be an action of a countable discrete group on a countable set. Consider a nonsingular Bernoulli action of $ G $ on $(\Omega_\Ber,\mu):=\prod_{x\in X_0} (\{0,1\},p_x\delta_0 + q_x \delta_1)$ that satisfies with Kakutani's condition.

\begin{itemize}

	\item (Almost periodic representation) Set $X:=X_0\sqcup IX_0$ and 
	$$H = \bigoplus_{x\in X_0} \C\eta_x\oplus \C I\eta_x = \ell^2(X),\quad \eta_x=\delta_x,\quad \eta_{Ix}=\delta_{Ix}\quad \text{for }x\in X_0.$$ 
Here $H$ has an involution $I$, and  $X$ has an order 2 bijection $I$, so that $If = \overline{f\circ I}$. The real part $H_\R\subset H$ is determined by $I$.
Set 
	$$A:=\bigoplus_{x\in X_0} \mu_x\oplus \mu_x^{-1},\quad  \text{on }H,\quad \text{where}\quad p_x= \frac{1}{1+\mu_{x}^{-1}},\quad q_x=\frac{\mu_{x}^{-1}}{1+\mu_{x}^{-1}}.$$
Define $U\colon \R \act H$ by $U_t:=A^{\ri t}$ for $t\in \R$. Since $U$ commutes with $I$, we get an almost periodic representation $U\colon \R\act H_\R$.

	\item (von Neumann algebra) We can define $\mathrm{C}^*_\anti(H_\R,U)\subset  \Gamma _\anti(H_\R,U)$ with the vacuum state $\varphi$. They are generated by 
\begin{align*}
	c_x=\frac{1}{\sqrt{2}}W(\widehat{\delta}_x)=\frac{1}{\sqrt{2}}(\ell(\widehat{\delta}_x)+\ell(\widehat{\delta}_{Ix})^*),\quad x\in X.
\end{align*}
The family $\{c_x\}_{x\in X_0}$ satisfies the CAR and we have an UHF isomorphism $\mathrm{C}^*_\anti(H_\R,U)\simeq \bigotimes_{x\in X_0}\M_2$. This sends $\varphi$ to the product state coming from $\mu$, hence it extends to a state preserving isomorphism from $ \Gamma_\anti(H_\R,U)$ to the Araki--Woods factor. This gives an embedding with expectation $E$
	$$L^\infty(\Omega_\Ber,\mu)\simeq \mathrm{W}^*\{c_x^*c_x,c_xc_x^*\mid x\in X_0\}\subset^E  \Gamma _\anti(H_\R,U).$$
Put $\cF_\anti := \cF_\anti (\ell^2(X))$ which is a standard representation for $ \Gamma _\anti(H_\R,U)$. 

	\item ($ G $-action) Let $ G  \act X=X_0\sqcup IX_0$ be the diagonal action and $ G \act^\pi \ell^2(X)=H$ the associated action. Since this commutes with $I$ and $P$, it induces an action on the CAR algebra $\mathrm{C}^*\{c_x\mid x\in X_0\} = \mathrm{C}^*_\anti(H_\R,U)$ by 
	$$\alpha_g(c_x) = c_{g\cdot x}\quad (\Leftrightarrow \alpha_g(W(\widehat{\delta}_x)) = W(\widehat{\delta}_{g\cdot x}) ),\quad \text{for }x\in X_0.$$
By Kakutani's condition, we can extend this action to $ \Gamma _\anti(H_\R,U)$. The resulting $\alpha$ restricts to the given Bernoulli action on $L^\infty(\Omega_\Ber,\mu)$.

\end{itemize}

\end{notation}
\begin{lem}\label{lem-RN-Strepn}
	Keep the setting from Notation \ref{notation}. For any $g\in  G $ there exists a unique $h_g\in L^1(\Omega_\Ber,\mu)$ (Radon--Nikodym derivative) such that $\varphi\circ \alpha_g^{-1}=\varphi(\, \cdot \, h_g)$ and
	$$ h_g  = \prod_{x\in X_0}\left(\frac{p_{g^{-1}\cdot x}}{p_{x}}c_{x}^*c_{x} + \frac{q_{g^{-1}\cdot x}}{q_{x}}c_{x}c_{x}^*\right) .$$
In particular, the standard implementation $U_g^\alpha$ of $\alpha_g$ on $\cF_\anti$ is given by 
	$$ U_g^\alpha\colon \cF_\anti \ni x\Omega \mapsto \alpha_g(x)h_g^{1/2}\Omega \in \cF_\anti ,\quad \text{for all }x\in  \Gamma _\anti(H_\R,U),\ g\in  G .$$
\end{lem}
\begin{proof}
	Since $\alpha_g$ preserves the expectation $E\colon  \Gamma _{\anti}(H_\R,U)\to L^\infty(\Omega_\Ber,\nu)$, if $h_g$ is the Radon--Nikodym derivative so that $\mu\circ \alpha_g^{-1} = \mu(\,\cdot\, h_g)$, then it satisfies $\varphi \circ \alpha_g^{-1}=\varphi(\,\cdot\, h_g)$. It is easy to see that $h_g$ is the form in the statement. The last part in the statement follows by the general theory.
\end{proof}

\subsection{How to use the boundary}\label{How to use the boundary}

Keep the action $ G  \act X_0$ and we use Notation \ref{notation}. For $ G  \act^\pi \ell^2(X)$, we can construct an embedding  $\ell^\infty(\sfZ)\subset \B(\cF_\anti)$ and a boundary action $ G  \act C(\partial \sfZ)$ as in Lemma \ref{lem-action-boundary2}, where the action is induced by $\Ad(\pi^\anti_g)$ for $g\in  G $. Then by Theorem \ref{thm-boundary-amenability}, the action is amenable if $ G $ is exact and $ G \act X_0$ has finite stabilizers and finitely many orbits. If the given Bernoulli action is measure preserving, we can follow Ozawa's proof \cite{Oz04} to deduce the solidity.

In our situation, however, the Bernoulli action can be non measure preserving. Then the standard implementation $U_g^\alpha$ of $ G  \act^\alpha  \Gamma_\anti(H_\R,U)$ does \textit{not} coincide with $\pi_g^\anti$, which yields several technical issues. For this reason, we can not follow the idea in \cite{Oz04}. 
For $g\in  G $, set 
	$$\mathrm{supp}(g):=\{i\in X_0\mid \mu_i\neq \mu_{g\cdot i}\}=\{i\in X_0\mid p_i\neq p_{g\cdot i}\}.$$
We study how to use $U_g^\alpha$ instead of $\pi_g^\anti$, under the assumption $|\mathrm{supp}(g)|<\infty$. 

The next lemma is the key observation in this direction. The lemma shows that $U_g^\alpha$ coincides with $\pi^\anti_g$ \textit{up to} a C$^*$-algebra. The proof of the lemma is the most technical part in the paper, so we prove it later in Subsection \ref{Proof of key lemma}. 

\begin{lem}\label{key lemma}
	Fix $g\in  G $ and assume $|\mathrm{supp}(g)|<\infty$. Then there exist $n\in \N$ and $X_i,Y_i\in \mathrm{C}^*\{\ell(x),r(y),f\mid x,y\in X,\ f\in \ell^\infty(\sfZ)\}$ for $i=1,\ldots,n$ such that
\begin{align*}
	 U^\alpha_g = \sum_{i=1}^nX_i \pi^\anti_g Y_i .
\end{align*}
\end{lem}

Here is a trivial lemma.

\begin{lem}\label{key lemma2}
Fix $g\in  G $ and assume $|\mathrm{supp}(g)|<\infty$. Then it holds that  
	$$U_g^{\alpha }fU_g^{\alpha *} \in  g\cdot f + \K(\cF_\anti),\quad \text{for any }f\in C(\overline{\sfZ}).$$
\end{lem}
\begin{proof}
	By Lemma \ref{key lemma}, write $U_g^\alpha = \sum_{i=1}^n X_i \pi^\anti_g Y_i$. Then by the definition of $C(\overline{\sfZ})$, $f$ commutes with $X_i,Y_i$ up to compacts. Since $\pi^\anti_g f = g\cdot f\,  \pi^\anti_g$, we get
	$$X_i\pi^\anti_g Y_i f = g\cdot f \, X_i\pi^\anti_g Y_i  + T_i\quad \text{for some }T_i\in \K(\cF_\anti)$$
for all $i$. This implies $U_g^\alpha f = g\cdot f \, U_g^\alpha + T_1+\cdots+T_n$ and get the conclusion.
\end{proof}

Recall from Subsection \ref{Standard representations} that $\cF_\anti \otimes \ell^2( G )$ is the standard representation of $ \Gamma _\anti(H_\R,U) \rtimes_\alpha  G $ and its commutant is also a crossed product $ \Gamma_{\anti,r}(H_\R,U) \rtimes_{\alpha^r}  G $. Here $ \Gamma_{\anti,r}(H_\R,U)$ is the von Neumann algebra generated by $W_r(I\widehat{\xi}) = JW(\widehat{\xi})J$ for $\xi\in H$. 
We regard $ \Gamma_{\anti,r}(H_\R,U) \rtimes_{\alpha^r}  G $ as a subalgebra in $\B(\cF_\anti \otimes \ell^2( G ))$. 
Set 
	$$ \cK:=\left\{\sum_{g\in  G }x_g\otimes e_{g,g}\in \B(\cF_\anti \otimes \ell^2( G )) \;\middle|\; x_g\in \K(\cF_\anti) \ \text{for all }g\in  G  \right\} $$
and observe that $\mathrm{C}_\anti^*(H_\R,U)\rtimes_{\alpha,{\rm red} } G $ and $\mathrm{C}_{\anti,r}^*(H_\R,U)\rtimes_{\alpha^r,{\rm red} } G  $ are contained in the multiplier $\mathrm{M}(\cK)$ of $\cK$.

Now we deduce condition (AO) in our setting. 

\begin{prop}\label{prop-AO}
	Keep the notation from Notation \ref{notation} and assume that $ G $ is exact and $ G \act X_0$ has finite stabilizers and finitely many orbits. Assume that $|\mathrm{supp}(g)|<\infty$ for all $g\in  G $. 
Then the following algebraic $\ast$-homomorphism is min-bounded:
\begin{align*}
	\nu\colon (\mathrm{C}_\anti^*(H_\R,U)\rtimes_{\alpha,{\rm red }} G ) \ota (\mathrm{C}_{\anti,r}^*(H_\R,U)\rtimes_{\alpha^r,{\rm red }} G ) &\to \mathrm{M}(\cK)/\cK
\end{align*}
given by $\nu(x\otimes y) = xy+\cK$.
\end{prop}
\begin{proof}
	For simplicity write $B:=\mathrm{C}_\anti^*(H_\R,U)$ and $B_r:=\mathrm{C}_{\anti,r}^*(H_\R,U)$. We denote by $\beta\colon  G  \act C(\partial \sfZ)$ the action in Lemma \ref{lem-action-boundary2}, hence $\beta$ is amenable by Theorem \ref{thm-boundary-amenability}. Since $c_0(\sfZ)\otimes 1 \subset \cK$, there is a $\ast$-homomorphism
	$$\pi_{\partial}\colon C(\partial \sfZ) \ni f+c_0(\sfZ) \mapsto f\otimes 1+\cK\in \mathrm{M}(\cK))/\cK.$$
We denote by $\pi_\ell,\pi_r$ the embedding of $B\rtimes_{\alpha,{\rm red}} G ,B_r\rtimes_{\alpha^r,{\rm red}} G $ into $\mathrm{M}(\cK)$. 
We claim that for any $f\in C(\overline{\sfZ})$, $x\in B$, $y\in B_r\rtimes_{\alpha^r,{\rm red}} G$,
	$$[f\otimes 1, \pi_\ell(x)],\ [f\otimes 1, \pi_r(y)]\in \cK.$$
To see this, as in the proof of Lemma \ref{lem-action-boundary1}, we have only to see the case of generators. Then the claim follows by definition of $\cK$. 
Thus there is a well defined algebraic $\ast$-homomorphism
\begin{align*}
	\pi_{\partial}\otimes \pi_\ell \otimes \pi_r \colon C(\partial{\sfZ}) \ota B \ota (B_r\rtimes_{\alpha^r,{\rm red}} G ) \to \mathrm{M}(\cK)/\cK.
\end{align*}
Since $B$ and $C(\partial{\sfZ})$ are nuclear, it is min-bounded. 

Consider the $ G $-action on $C(\partial \sfZ)\otm B \otm (B_r\rtimes_{\alpha^r,{\rm red}} G )$, which is the diagonal action of $\beta$, $\alpha$, and the trivial action. 
Consider the $ G $-action on $\mathrm{M}(\cK)/\cK$ given by $\Ad(U_g^\alpha\otimes \lambda_g)$ for $g\in  G $. Then it restricts to $\alpha$ on $\pi_\ell(B)$ and the trivial action on $\pi_r(B_r\rtimes_{\alpha^r,{\rm red}} G )$. We claim that it restricts to $\beta$ on $\pi_{\partial}(C(\partial\sfZ))$.

	To see the claim, it is enough to see that $\Ad(U_g^\alpha)$ for $g\in  G $ coincides with $\beta$ on $C(\partial\sfZ)$. 
Observe by Proposition \ref{prop-embedding} that
	$$ \frac{C(\overline{\sfZ})+\K(\cF_\anti)}{\K(\cF_\anti)} = \frac{C(\overline{\sfZ})}{C(\overline{\sfZ})\cap \K(\cF_\anti)} = \frac{C(\overline{\sfZ})}{c_0(\sfZ)} =C(\partial\sfZ).$$
Then by Lemma \ref{key lemma2}, $\Ad(U_g^\alpha)$ induces an automorphism of $C(\partial\sfZ)$ which is given by $f+\K(\cF_\anti)\mapsto g\cdot f + \K(\cF_\anti)$. This is exactly $\beta$ on $C(\partial \sfZ)$ and the claim is proven.

Thus we conclude that $\pi_{\partial}\otimes \pi_\ell \otimes \pi_r$ is $ G $-equivariant. Since $ G  \act^\beta C(\partial \sfZ)$ is amenable and since $C(\partial \sfZ)$ is contained in the center of $C(\partial \sfZ)\otm B\otm (B_r\rtimes_{\alpha^r,{\rm red}} G  )$, it follows that the map
\begin{align*}
	[C(\partial \sfZ) \otm B \otm (B_r\rtimes_{\alpha^r,{\rm red}} G )]\rtimes_{\rm red} G  \to \mathrm{M}(\cK)/\cK
\end{align*}
is bounded. Finally the domain contains
	$$[B \otm (B_r\rtimes_{\alpha^r,{\rm red}} G )]\rtimes_{{\rm red}} G   = (B\rtimes_{\alpha,{\rm red}} G )\otm(B_r\rtimes_{\alpha^r,{\rm red}} G )$$
and the restriction to it coincides with $\nu$. Thus $\nu$ is bounded.
\end{proof}

\begin{proof}[Proof of Proposition \ref{propE}]
	Keep the setting in Notation \ref{notation} and let $e\colon \cF_\anti \to L^2(\Omega_\Ber,\mu)$ be the orthogonal projection. Then since $e\otimes 1\in \mathrm{M}(\cK)$, the compression map by $e\otimes 1$ induces a ucp map from $\mathrm{M}(\cK)/\cK $ into $e\mathrm{M}(\cK)e/e\cK e$ (which coincides with the range of $\nu$ in Proposition \ref{propE}). By restricting $\nu$ in Proposition \ref{prop-AO} and by composing it with the compression map, we get the conclusion.
\end{proof}

\subsection{Proof of Theorem \ref{thmA}}
\label{Proof of Theorem A}

We arrive at the main theorem of this article. 
Combining this theorem with a result in the appendix, we will obtain  solidity of actions, Note that the freeness is not necessary in this theorem. Recall that a von Neumann algebra with separable predual is called \textit{amenable} if it is hyperfinite (which has many different characterizations \cite{Co75}).

\begin{thm}\label{solid-Bernoulli-thm}
	Keep the setting as in Notation \ref{notation}. Assume that $ G $ is exact and $ G \act X_0$ has finite stabilizers and finitely many orbits. 
Assume that $|\mathrm{supp}(g)|<\infty$ for all $g\in  G $. Consider the associated action $G\act \Gamma_\anti(H_\R,U)=:M$ and the crossed product $M\rtimes G$.

Then for any projection $p\in M$ and any diffuse von Neumann subalgebra $A\subset p Mp$ with expectation, $A' \cap p( M\rtimes  G )p$ is amenable. Further if $ G  $ is bi-exact, then $q( M \rtimes  G )q$ is solid for any projection $q\in M\rtimes  G $.
\end{thm}
\begin{proof}
	For the first half of the statement, we mostly follow the proof of \cite[Theorem 4.7]{Oz04}. We give a sketch for the reader's convenience. Let $u$ be a Haar unitary in $A$ and we may assume $A$ is generated by $u$. Define a ucp map
	$$ \Phi\colon \B(\cF_\anti\otimes \ell^2( G ))\to  A' \cap p\B(\cF_\anti\otimes \ell^2( G ))p;\quad X\mapsto \lim_{n\to \omega}\frac{1}{n}\sum_{k=1}^n (u^k\otimes 1)X(u^{k *}\otimes 1),$$
where $\omega$ is a fixed free ultrafilter on $\N$. 
Let $\psi$ be any faithful normal state on $A$ and extend it on $M$ by the given expectation for $A\subset pMp$. We further extend it on $M\rtimes  G $ by the canonical expectation. Then $\psi\circ \Phi|_{M\rtimes  G }=\psi(p\, \cdot \, p)$. In particular $\Phi|_{p(M\rtimes  G )p}$ coincides with the unique $\psi$-preserving expectation $E_{A'}\colon p(M\rtimes  G )p \to A' \cap p(M\rtimes  G )p$. Once we get this condition, we can follow the proof of \cite[Theorem 4.7]{Oz04}.


	Assume next that $ G $ is bi-exact. We need some known results from Popa's deformation theory, so we only explain how to apply them. 
Suppose by contradiction that there exist a projection $p\in M\rtimes  G $, a diffuse von Neumann subalgebra $A\subset p(M\rtimes  G )p$ with expectation such that $P:=A' \cap p(M\rtimes  G )p$ is not amenable. We may assume $A$ is abelian. Up to cutting by a projection in $P$, we may assume that $P$ has no amenable direct summand. 
Then by \cite[Theorem 4.1]{HI15b}, we have $P'\cap p(M\rtimes  G )^\omega p\subset p(M^\omega\rtimes  G )p$, so that $A^\omega p\subset p(M^\omega\rtimes  G )p$. Combined with \cite[Theorem 4.3(4)]{HI15a}, this implies Popa's embedding condition $A \preceq_{M\rtimes  G } M$ (see \cite[Definition 4.1]{HI15a}). Thus there are projections $e\in A$, $f\in M$, a partial isometry $v\in e(M\rtimes  G )f$ and an injective $\ast$-homomorphism $\theta\colon eAe\to fMf$ such that $v\theta(a) = a v$ for all $a\in eAe$. By the first half of the proof, we know that $\theta(eAe)' \cap f(M\rtimes  G )f$ is amenable. In particular
	$$v^*v [\theta(eAe)' \cap f(M\rtimes  G )f ]v^*v = v^* [A'\cap p(M\rtimes  G )p ]v $$
is amenable, hence $A'\cap p(M\rtimes  G )p$ has an amenable direct summand, a contradiction.
\end{proof}

We prove Theorem \ref{thmA} and Corollary \ref{corD} for generalized Bernoulli actions as follows. 

\begin{thm}\label{main thm}
	Let $ G $ be a countable discrete group acting on a countable set $\cI$. Consider a product measure space with two base points
	$$ (\Omega_\Ber , \mu):=\prod_{i\in \mathcal I} (\{0,1\}, p_i\delta_0 + q_i \delta_1)\quad \text{where $p_i\in (0,1)$ and $p_i+q_i=1$ for all $i\in \cI$}, $$
which satisfies Kakutani's condition, so that the nonsingular (generalized) Bernoulli action $ G  \act (\Omega_\Ber,\mu)$ is defined. Assume that $(\Omega_\Ber,\mu)$ has no atoms and $G\act (\Omega_\Ber,\mu)$ is essentially free. Assume further that
\begin{itemize}
	\item[$\rm (i)$] $ G $ is exact and the action $ G  \act \cI$ has finite stabilizers and finitely many orbits;
	\item[$\rm (ii)$] for any $g\in  G $, $p_i = p_{g\cdot i}$ for all but finitely many $i\in \cI$. 
\end{itemize}
Then the Bernoulli action $ G  \act (\Omega_\Ber,\mu)$ is solid. Further $L^\infty(\Omega_\Ber,\mu)$ is solid if $G$ is bi-exact.
\end{thm}
\begin{proof}
	Putting $X_0:=\cI$, by Lemma \ref{lem-Bernoulli-embedding}, we can construct $G\act^\alpha \Gamma _\anti(H_\R,U)$ and we have only to show that $\alpha$ is 
solid. By construction, $|\supp(g)|<\infty$ for all $g\in  G $, hence we can apply Theorem \ref{solid-Bernoulli-thm} to $\alpha$, so that it is solid.

If $ G $ is bi-exact, then 
$ \Gamma_\anti(H_\R,U)\rtimes  G $ is solid, hence its subalgebra $L^\infty(\Omega_\Ber,\mu)\rtimes  G $ with expectation is also solid.
\end{proof}

\subsection{Proof of Lemma \ref{key lemma}}\label{Proof of key lemma}

In this subsection, we prove Lemma \ref{key lemma}. We keep the notation from Notation \ref{notation}.

We first recall the following well known fact: for any $\xi_1,\ldots,\xi_n\in H$, there is a unique $W(\widehat{\xi}_1\wedge \cdots \wedge \widehat{\xi}_n)\in  \Gamma_\anti(H_\R,U)$ such that 
\[
	W(\widehat{\xi}_1\wedge \cdots \wedge \widehat{\xi}_n)\, \Omega = \widehat{\xi}_1\wedge \cdots \wedge \widehat{\xi}_n .
\]
Note that this notation generalizes our previous one $W(\widehat{\xi})$ for $\xi \in H$, which corresponds to the case $n=1$. We have the following formula (see for example \cite[Lemma 1.3]{Hi02} whose proof works for $q=-1$). 
\begin{thm}
For any $n\in \N$ and $\xi_1,\ldots,\xi_{n}\in H$, we have
\begin{align}
	W(\widehat{\xi}_1)\cdots W(\widehat{\xi}_{n})
	=\sum_{\mathcal W}C(\mathcal W)\left(\prod_{r=1}^{\ell}\langle \widehat{I\xi}_{i(r)},\widehat{\xi}_{j(r)}\rangle\right)W(\widehat{\xi}_{k(1)}\wedge \cdots \wedge \widehat{\xi}_{k(m)}), \label{wick}
\end{align}
where the summation is over all partitions $\mathcal W=\{ \{i(r),j(r)\}_{1\leq r\leq \ell}, \{k(p)\}_{1\leq p\leq m} \}$ of $\{1,2,\ldots,n\}$ having blocks of one or two elements such that
\[
	\ell,m\geq 0,\quad 2\ell+m=n,\quad i(r)<j(r)\quad \text{for } 1\leq r\leq \ell,\ k(1)<\cdots<k(m),
\]
and $C(\cW):=\frac{1}{\sqrt{m!}}(-1)^{c(\cW)}$, where $c(\mathcal W)$ is given by 
\begin{align*}
	 c(\mathcal W)
	&=\sharp\{(r,s)\mid 1\leq r<s\leq \ell,\ i(r)<i(s)<j(r)<j(s)\} \\
	&\quad + \sharp\{(r,p)\mid 1\leq r\leq \ell,\ 1\leq p \leq m,\ i(r)<k(p)<j(r)\}.
\end{align*}
\end{thm}

Recall that $\pi_g^\anti$ is given by 
	$$ \pi_g^\anti \, (\delta_{x_1}\wedge \cdots\wedge \delta_{x_n}) = \delta_{g\cdot x_1}\wedge \cdots\wedge \delta_{g\cdot x_n},\quad  \text{for any distinct }x_1,\ldots,x_n\in X.$$
To compare it with $U_g^\alpha$, it is convenient to use a new operator $V_g$ for $g\in  G $ defined by
	$$ V_g\, (\widehat{\delta}_{x_1}\wedge \cdots\wedge \widehat{\delta}_{x_n}):= \widehat{\delta}_{g\cdot x_1}\wedge \cdots\wedge \widehat{\delta}_{g\cdot x_n},\quad \text{for any distinct }x_1,\ldots,x_n\in X.$$
This is defined on the algebraic level.

The next lemma is straightforward.

\begin{lem}\label{RN-bdd-lem}
	Fix $g\in  G $ and assume $|\mathrm{supp}(g)|<\infty$.
\begin{enumerate}
	\item The Radon--Nikodym $h_g$ is contained in $\mathrm{C}^*\{ c_x^*c_x, c_xc_x^*\mid x\in X_0 \}$.

	\item There is an invertible $f\in \ell^\infty(\sfZ)^+$ such that $V_g = f \pi^\anti_g$. In particular $V_g$ is bounded and invertible in $\B(\cF_\anti)$.

\end{enumerate}
\end{lem}
\begin{proof}
	1. This is trivial by the formula of $h_g$ given in Lemma \ref{lem-RN-Strepn}.

	2. Put $d(x) := \frac{\sqrt{2}}{\sqrt{1+\mu_x^{-1}}}$ and $d(Ix):=\frac{\sqrt{2}}{\sqrt{1+\mu_x}}$ for $x\in X_0$, and observe that $\widehat{\delta}_x = d(x)\delta_x$ for all $x\in X$. For any distinct $x_1,\ldots,x_n\in X$, with $d:=d(x_1)\cdots d(x_n)$ and $d_g:=d(g\cdot x_1) \cdots d(g\cdot x_n)$, it is easy to see that
\begin{align*}
	V_g(\delta_{x_1}\wedge \cdots \wedge \delta_{x_n})
	=d^{-1}d_g \pi_g^\anti (\delta_{x_1}\wedge \cdots \wedge \delta_{x_n}).
\end{align*}
Hence if we put $f([x_1,\ldots,x_n])=d^{-1}d_g$, then since $|\supp(g)|<\infty$, $f$ defines an invertible element in $ \ell^\infty(\sfZ)^+$.
\end{proof}

	We are interested in the difference between $U_g^\alpha$ and $V_g$. For $x\in X$, it holds that 
\begin{align*}
	 U_g^\alpha \widehat{\delta}_x 
	&= U_g^\alpha W(\widehat{\delta}_x)\Omega 
	=  \alpha_g(W(\widehat{\delta}_x))h_g^{1/2}\Omega \\
	&=  Jh_g^{1/2}J W(\widehat{\delta}_{g\cdot x})\Omega 
	=Jh_g^{1/2}J \widehat{\delta}_{g\cdot x}\\
	&=Jh_g^{1/2}J V_g\widehat{\delta}_{x}.
\end{align*}
This means $U_g^\alpha = Jh_g^{1/2}JV_g$ on $\ell^2(X)$. To prove Lemma \ref{key lemma}, we need a similar result on $\ell^2(X)^{\wedge n}$, but it is not easy because we do not know the behavior of $U_g^\alpha$ on $\ell^2(X)^{\wedge n}$. Indeed, it often happens that
\begin{align*}
	 \alpha_g( W(\widehat{\delta}_{x_1}\wedge \cdots \wedge \widehat{\delta}_{x_n})) \neq W(\widehat{\delta}_{g\cdot x_1}\wedge \cdots \wedge \widehat{\delta}_{g\cdot x_n})
\end{align*}
and therefore the above computation does not work. We have to discuss this problem.

\begin{notation}\upshape
For notation simplicity, we identify $\delta_x\in \ell^2(X)$ and $x\in X$, and write $gx:=g\cdot x$ for $g\in  G $ and $x\in X$. For example, $V_g$ can be written as
	$$V_g (\widehat{  x}_1  \wedge  \dots \wedge \widehat { x}_n  )= 
\widehat{ g x}_1  \wedge  \dots \wedge \widehat {g x}_n  .$$
\end{notation}

\begin{lem}\label{lem-forkeylemma1}
For any $g \in  G $, and $F = \{ x_1, \dots, x_n \} \subset X$ such that $F\cap IF  \cap \supp (g) = \emptyset$,
it follows that
	$$(U_g^\alpha  - J h_g^{1/2} J V_g) (\widehat{ x }_1 \wedge \dots \wedge \widehat { x }_n) = 0.$$
\end{lem}
\begin{proof}
We prove the assertion by induction on $n$. The case $n=0$ is  trivial and the case $n=1$ was observed above. 
For general $n \geq 2$, we denote the partition of $\{1,\ldots,n\}$ into $n$ singletons by $\cW_n$.
Then, by Equation (\ref{wick})
one has
\begin{align*}
	 V_g ( \widehat{ x }_1 \wedge \dots \wedge \widehat{ x }_n )
	&	= (\widehat{ gx }_1 \wedge \dots \wedge \widehat{ gx }_n )\\
	&	= W(\widehat{gx}_1)\cdots W(\widehat{gx}_n) \Omega \\
	&\quad -\sum_{\cW = \{i,j,k \} \neq \cW_n} C(\cW)
		\left(
			\prod_{ r = 1 }^{ \ell } \langle \widehat{Igx}_{i(r)},\widehat{gx}_{j(r)}\rangle
		\right)
 		 V_g (\widehat{ x }_{ k(1) } \wedge \dots \wedge \widehat { x }_{ k(m) } ).
\end{align*}
We next apply $Jh_g^{-1/2}JU_g^\alpha $ to Equation (\ref{wick}) (with $\Omega$), then
\begin{align*}
	Jh_g^{-1/2}J U_g^\alpha (\widehat{ x }_1 \wedge \dots \wedge \widehat { x }_n )
	&= W(\widehat{gx}_1)\cdots W(\widehat{gx}_n) \Omega \\
	& \quad -\sum_{\mathcal W = \{i,j,k \} \neq \cW_n} C(\cW)
		\left(
			\prod_{ r = 1 }^{ \ell } \langle \widehat{Ix}_{i(r)},\widehat{x}_{j(r)}\rangle
		\right)
 		J h_g^{ -1/2 } J  U_g^\alpha (\widehat{ x }_{k(1)} \wedge \dots \wedge \widehat { x }_{ k(m) } ),
\end{align*}
where we used the equation 
	$$Jh_g^{-1/2}J U_g^\alpha W(\widehat{x}_1)\cdots W(\widehat{x}_n) \Omega = \alpha_g(W(\widehat{x}_1))\cdots \alpha_g(W(\widehat{x}_n)) \Omega = W(\widehat{gx}_1)\cdots W(\widehat{gx}_n) \Omega.$$

Fix $\cW = \{ i,j,k \} \neq \cW_n$.
By our assumption on $F$, if $Ix_{ i(r) } = x_{ j(r) }$ holds for some $r$, then $x_{ i(r) }, x_{ i(r) } \notin \supp (g) \cup I \supp (g)$.
This implies that $\langle \widehat{Ix}_{i(r)},\widehat{x}_{j(r)}\rangle = \langle \widehat{Igx}_{i(r)},\widehat{gx}_{j(r)}\rangle$ holds for all $r$ (to see this, compare scalars $d(y)>0$ given by $\widehat{\delta}_y = d(y)\delta_y$ for $y\in X$). 
Hence, letting $D_\cW := C(\cW)\prod_{ r = 1 }^{ \ell } \langle \widehat{Ix}_{i(r)},\widehat{x}_{j(r)}\rangle= C(\cW)\prod_{ r = 1 }^{ \ell } \langle \widehat{Igx}_{i(r)},\widehat{gx}_{j(r)}\rangle$, we get
\[
	(J h_g^{ -1/2 } J U_g^\alpha  - V_g) (\widehat{ x }_1 \wedge  \dots \wedge \widehat { x }_n )
	= \sum_{\mathcal W = \{i,j,k \} \neq \cW_n} D_\cW ( V_g - J h_g^{ -1/2 } J U_g^\alpha) (\widehat{ x }_{ k(1) } \wedge \dots \wedge \widehat { x }_{k(m)} ).
\]
Observe that in this equation, any $m$ appearing in the right hand side satisfies $m \leq n-2$, hence by the hypothesis of the induction, the right hand side is $0$. We get the conclusion of the induction on $n$.
\end{proof}
Fix $g \in  G $.
For any finite subset $F \subset X$ we set
\[
	A_{g, F} := \{ x \in F \mid x \in \supp (g), I x \in F \} = \supp(g) \cap F \cap IF.
\]
The previous lemma shows that $U_g^\alpha = Jh_g^{1/2}JV_g$ holds on vectors coming from $F$ such that $A_{g,F}=\emptyset$. We have to study the case $A_{g,F}\neq \emptyset$. 

For any finite subset $F = \{x_1, \dots, x_n \} \subset \supp (g)$, define $K_{g, F}$ as the closed subspace of $\cF_\anti$ generated by
\begin{align*}
	x_1 \wedge I x_1 \wedge \dots \wedge x_n \wedge I x_n \wedge y_1 \wedge \dots \wedge y_m ,
\end{align*}
for $m \geq 0$, distinct $y_1,\ldots,y_m\in X \setminus (F\cup IF)$ such that $A_{g,\{y_1,\ldots,y_m\}} = \emptyset $. Here $m=0$ means $y_1 \wedge \dots \wedge y_m = \Omega$ and we use the convention $\eta\wedge\Omega = \eta$ for all $\eta$. Note that $F=\emptyset$ also makes sense (i.e.\ $x_k,Ix_k$ do not appear). We denote by $P_{g,F}$ the orthogonal projection onto $K_{g, F}$. 

\begin{lem}\label{lem-forkeylemma2}
For any $g\in G$, it holds that
	$$1=\sum_{F\subset \supp(g)}P_{g,F},\quad \text{where $F=\emptyset$ is included}.$$
If $|\supp(g)|<\infty$, then $P_{g,F}$ belongs to $\mathrm{C}^*\{ \ell(x) \mid x\in X \}$ for all subsets $F\subset \supp(g)$. 
\end{lem}
\begin{proof}
	For any $x \in X_0$, set
\begin{align*}
	v_x := \ell(x) \ell (Ix),\quad \quad
	w_x := \ell (x) \ell (Ix)^* + \ell (Ix) \ell (x)^* + \ell (x)^*\ell(Ix)^*.
\end{align*}
Then $v_x$ is a partial isometry whose range is generated by $x\wedge Ix \wedge y_1\wedge \cdots \wedge y_m$ for all $m\geq0$ and distinct $y_1,\ldots,y_m\in X$ such that $x,Ix\not\in \{y_1,\ldots,y_m\}$. It is easy to see that $w_x$ is also a partial isometry such that $v_xv_x^*+w_xw_x^*=1$. Since the family $\{ v_xv_x^*\}_{x \in X_0}$ mutually commute, it follows that
\[
	1
	= \prod_{ x \in \supp (g) } (v_x v_x^* + w_x w_x^* )
	= \sum_{ F \subset \supp (g) } \left( \prod_{ x \in F } v_x v_x^* \prod_{ y \in \supp (g) \setminus F } w_yw_y^* \right),
\]
where $F$ is possibly empty. Since we know the range of $v_xv_x^*$ and $w_xw_x^*=1-v_xv_x^*$ for $x\in X$, it is straightforward to show that
	$$P_{g, F} = \prod_{x \in F} v_x v_x^*  \prod_{y \in \supp (g) \setminus F} w_yw_y^* \quad \text{for all }F\subset \supp(g).$$
If $|\supp (g)|<\infty$, it is contained in $\mathrm{C}^*\{\ell(x)\mid x\in X\}$.
\end{proof}

\begin{lem}\label{lem-forkeylemma3}
	Let $F = \{ x_1, \dots, x_n \} \subset \supp(g)$ be a finite subset.
Then there exists an invertible element $Z_F \in \mathrm{C}^*(H_\R,U)$  such that, with $r_{F,m}:=\sqrt{m+1}\cdots\sqrt{m+2n}$, 
$$
	Z_F ( y_1  \wedge \dots \wedge y_m )
	= 
	r_{F,m}\, ( \widehat{ x }_1 \wedge \widehat{ Ix }_1 ) \wedge \dots \wedge ( \widehat{ x }_n \wedge \widehat{ Ix }_n ) \wedge
	( y_1 \wedge \dots \wedge y_m ),
$$
for all $m \geq 0$ and $y_1, \dots, y_m  \in X\setminus (F\cup IF)$ such that $A_{g,\{y_1,\ldots,y_m\}}=\emptyset$ (where $m=0$ means $y_1\wedge \cdots \wedge y_m=\Omega$).
\end{lem}
\begin{proof}
	Put $d(x) := \frac{\sqrt{2}}{\sqrt{1+\mu_x^{-1}}}$ and $d(Ix):=\frac{\sqrt{2}}{\sqrt{1+\mu_x}}$ for $x\in X_0$, so that $\widehat{\delta}_x = d(x)\delta_x$ for all $x\in X$. It holds that
\begin{align*}
	c_x=\frac{1}{\sqrt{2}}W(\widehat{\delta}_x)=\frac{1}{\sqrt{2}}(d(x)\ell(\delta_x)+d(Ix)\ell(\delta_{Ix})^*),\quad x\in X_0.
\end{align*}
Since $c_x^*=c_{Ix}$ for $x\in X_0$, the same equality holds for all $x\in X$. Fix $x:=x_i\in F$. For any given $\xi := y_1 \wedge \dots \wedge y_m $ in the statement, we have (with $r_k:=\sqrt{k}$)
\begin{align*}
	c_x c_x^* \xi
	& = \frac{1}{2} W(\widehat{x}) r_{m+1}\, d( Ix) Ix \wedge \xi
	= \frac{1}{2} (r_{m+1}r_{m+2}\, \widehat{ x } \wedge \widehat{ Ix } \wedge \xi +  d(Ix)^2 \xi) ; \\
	c_x^* c_x \xi
	& = \frac{1}{2} W(\widehat{x})^*r_{m+1} d(x) x \wedge \xi
	=  \frac{1}{2} (-r_{m+1}r_{m+2}\, \widehat{ x } \wedge \widehat{ Ix } \wedge \xi + d(x)^2 \xi).
\end{align*}
It follows that 
\begin{align*}
	\left(\frac{1}{d(Ix)^2}c_x c_x^*  - \frac{1}{d(x)^2}c_x^*c_x\right)\xi = \frac{r_{m+1}r_{m+2}}{2}\left( \frac{1}{d(Ix)^2}+\frac{1}{d(x)^2} \right)\widehat{ x } \wedge \widehat{ Ix } \wedge \xi.
\end{align*}
Since $\left( \frac{1}{2d(Ix)^2}+\frac{1}{2d(x)^2} \right)^{-1}=d(x)^2d(Ix)^2$ (e.g.\ $d(Ix)=\frac{\sqrt{2\mu_x^{-1}}}{\sqrt{1+\mu_x^{-1}}}$), if we put
	$$Z_x:=d(x)^2c_x c_x^*  - d(Ix)^2c_x^*c_x\in \mathrm{C}^*(H_\R,U),$$
then $Z_x\xi = r_{m+1}r_{m+2}\, (\widehat{ x } \wedge \widehat{ Ix } \wedge \xi)$. 
Note that $Z_x$ is invertible since $c_xc_x^*$ and $c_x^*c_x$ are projections with $c_xc_x^*+c_x^*c_x=1$. Finally it is easy to see that  
\[
	Z_F :=  \prod_{x \in F } Z_x\in \mathrm{C}^*(H_\R,U).
\]
gives the desired element.
\end{proof}

Now we prove the key lemma.

\begin{proof}[Proof of Lemma \ref{key lemma}]
	Fix $g \in  G $ and any nonempty set $F  = \{x_1, \dots, x_n \} \subset \supp (g)$. Take $Z_F$ as in Lemma \ref{lem-forkeylemma3} and observe that $\alpha_g(Z_F)=U_g^\alpha Z_F U_g^{\alpha *}$ since $Z_F\in  \mathrm{C}^*_\anti(H_\R,U)$. We claim that 
	$$ U_g^\alpha  P_{g,F} = J h_g^{1/2}J \alpha_g (Z_F) V_g Z_F^{ -1 } P_{ g, F }.$$
To see this, take any $y_1,\ldots,y_m\in X\setminus ( F\cup IF )$ such that $A_{g,\{y_1,\ldots,y_m\}}=\emptyset$ and put $\xi:=y_1\wedge \cdots \wedge y_m$ (possibly $m=0$ and $\xi=\Omega$). Then, we compute that 
\begin{align*}
	&  U_g^\alpha 
		(\widehat{ x }_1 \wedge \widehat{ Ix }_1 ) \wedge \dots \wedge (\widehat{ x }_n \wedge \widehat{ Ix }_n )
		\wedge \xi \\
	& \quad = r_{F,m}^{-1} U_g^\alpha  Z_F \,  \xi  \quad \quad \quad \quad \quad \quad \quad \text{(by Lemma \ref{lem-forkeylemma3})}\\
	& \quad = r_{F,m}^{-1} \alpha_g(Z_F) U_g^\alpha \, \xi \\
	& \quad = r_{F,m}^{-1} \alpha_g(Z_F) J h_g^{1/2}JV_g \, \xi  \quad \quad\quad  \text{(by Lemma \ref{lem-forkeylemma1})}\\
	& \quad = r_{F,m}^{-1} J h_g^{1/2}J\alpha_g(Z_F) V_g Z_F^{-1} Z_F \, \xi \\
	& \quad = J h_g^{1/2}J\alpha_g(Z_F) V_g Z_F^{-1} (\widehat{ x }_1 \wedge \widehat{ Ix }_1 ) \wedge \dots \wedge (\widehat{ x }_n \wedge \widehat{ Ix }_n )
		\wedge \xi.
\end{align*}
Since such $(\widehat{ x }_1 \wedge \widehat{ Ix }_1 ) \wedge \dots \wedge (\widehat{ x }_n \wedge \widehat{ Ix }_n )\wedge \xi$ generate the range of $P_{g,F}$, the claim is proven.

Since $U_g^\alpha P_{g,\emptyset} = Jh_g^{1/2}J V_g P_{g,\emptyset}$ by Lemma \ref{lem-forkeylemma1}, combined with Lemma \ref{lem-forkeylemma2}, we have
\[
	U_g^\alpha = J h_g^{1/2} J V_gP_{g,\emptyset} + \sum_{\emptyset \neq F \subset \supp (g) } J h_g^{1/2} J \alpha_g (Z_F) V_g Z_F^{-1} P_{g, F}.
\]
Since the sum is finite and by Lemma \ref{RN-bdd-lem}, this is the desired decomposition. 
\end{proof}

\subsection{Boundary for Bernoulli actions}
\label{Boundary for Bernoulli actions}

We keep the setting from Notation \ref{notation}. We have constructed a boundary $C(\partial \sfZ)$ which leads condition AO for the crossed product $\mathrm{C}^*_\anti(H_\R,U)\rtimes_{\rm red}  G $, as in Proposition \ref{prop-AO}. Since $\mathrm{C}^*_\anti(H_\R,U)\rtimes_{\rm red}  G $ contains a copy of $C(\Omega_\Ber)\rtimes_{\rm red} G  $, it is natural to ask if we have a boundary constructed on $L^2(\Omega_\Ber,\mu)$. In this last subsection, we explain a construction of such a boundary. 

Recall that we have a state preserving embedding with expectation $E$
	$$ L^\infty(\Omega_\Ber,\mu) \simeq \mathrm{W}^*\{ c_x^*c_x \mid x\in X \}\subset^E  \Gamma_\anti(H_\R,U). $$
This induces a natural embedding $L^2(\Omega_\Ber,\mu)\subset \cF_\anti$. We denote by $e\colon \cF_\anti \to L^2(\Omega_\Ber,\mu)$ the orthogonal projection, which is the Jones projection for $E$. Then $J_\Ber:=J e$ is the modular conjugation for $L^\infty(\Omega_\Ber,\mu)$.

\begin{lem}
	In this setting $L^2(\Omega_\Ber,\mu)$ is generated by $\Omega$ and
	$$ x_1\wedge Ix_1 \wedge x_2\wedge Ix_2 \wedge \cdots \wedge x_n\wedge Ix_n,\quad \text{for all }n\in \N \text{ and distinct } x_1,\ldots,x_n\in X .$$
\end{lem}
\begin{proof}
	Since $L^2(\Omega_\Ber,\mu)$ coincides with the closure of $\mathrm{W}^*\{ c_x^*c_x \mid x\in X \}\Omega$, the lemma follows from direct computations.
\end{proof}

Recall $Z_n=\{[x_1,\ldots,x_n]\in X^n/\frakS_n\mid x_i\neq x_j\ \text{for all }i\neq j\}$. Define $W_n$ as the set of all $[x_1,Ix_1,x_2,Ix_2,\ldots,x_{n},Ix_n]$ which defines an element in $Z_{2n}$. Set
	$$ \sfW:=\{\star\}\sqcup \bigsqcup_{n\geq 1}W_n \subset \sfZ$$
and observe that $\ell^\infty(\sfW) = \ell^\infty(\sfZ)e$. We have a natural inclusion 
	$$\ell^\infty(\sfW)= \ell^\infty(\sfZ)e \subset e\B(\cF_\anti)e = \B(L^2(\Omega_\Ber,\mu)).$$
In this way, we have an embedding $\ell^\infty(\sfW)\subset  \B(L^2(\Omega_\Ber,\mu))$. Set
\begin{align*}
	\cC_{\ell,{\rm Ber}}:=\mathrm{C}^*\{ \ell(x)\ell(Ix)e \mid x\in X\} \subset \B(L^2(\Omega_\Ber,\mu))
\end{align*}
and $\cC_{r,\Ber}:=J_\Ber \cC_{\ell,\Ber}J_\Ber$.

\begin{lem}
	It follows that $C(\Omega_\Ber)\subset \cC_{\ell,\Ber}$.
\end{lem}
\begin{proof}
For $x\in X$, $2c_xc_x^*$ coincides with
\begin{align*}
	 d(x)d(Ix)\ell(x)\ell(Ix) + d(x)^2\ell(x)\ell(x)^*
+ d(Ix)^2\ell(Ix)^*\ell(Ix) 
+ d(Ix)d(x)\ell(Ix)^*\ell(x)^*,
\end{align*}
where $d(y)>0$ for $y\in X$ is given by $\widehat{\delta}_y=d(y)\delta_y$. 
To see $c_xc_x^*e \in \cC_{\ell,\Ber}$, we have only to show $\ell(x)\ell(x)^*e,\ell(Ix)^*\ell(Ix)e \in \cC_{\ell, \Ber}$. Observe that $\ell(x)\ell(x)^*e$ is the projection onto the space generated by words containing both of $x$ and $Ix$. It holds that
	$$\ell(x)\ell(x)^* e = \ell(Ix)\ell(Ix)^*\ell(x)\ell(x)^* e = \ell(Ix)\ell(x)e (\ell(Ix)\ell(x) e)^* \in \cC_{\ell,\Ber}.$$
Since $x\in X$ is arbitrary, this finishes the proof.
\end{proof}

Set
	$$ C(\overline{\sfW}) := \{ f\in \ell^\infty(\sfW)\mid [f,a]\in \K(L^2(\Omega_\Ber,\mu)),\ \text{for all }a\in \cC_{\ell,\Ber}\cup \cC_{r,\Ber} \}.$$
and define a boundary C$^*$-algebra by $C(\partial \sfW):=C(\overline{\sfW})/c_0(\sfW)$. It is easy to see that $ G  \act X$ induces an action $ G  \act C(\partial \sfW)$, as in Lemma \ref{lem-action-boundary2}. The next proposition shows that the boundary amenability holds for $C(\partial \sfW)$.

\begin{prop}
	If $ G $ is exact and $ G  \act X_0$ has finite stabilizers and finitely many orbits, then the action $ G  \act C(\partial \sfW)$ is amenable.
\end{prop}
\begin{proof}
	Let $\mu\colon \sfZ \to \mathrm{Prob}(X)$ be as in Subsection \ref{Definition of boundary}. Consider its restriction to $\sfW$ and induce $\mu^*\colon \ell^\infty(X)\to \ell^\infty(\sfW)$. Then by restricting $\sfW\subset \sfZ$, the proof of Lemma \ref{lem-boundary-amenability} shows that 
\begin{align*}
	\mu^* (g \cdot \varphi ) - g\cdot \mu^*( \varphi ) \in c_0(\sfW),\quad [\mu^*(\varphi),\ell(x)\ell(Ix)e],\ [\mu^*(\varphi),r(x)r(Ix)e]\in \K(L^2(\Omega_\Ber,\mu)).
\end{align*}
for all $g \in  G $, $\varphi \in \ell^\infty (X)$ and $x\in X$. This gives the conclusion.
\end{proof}

Using this proposition, we can directly deduce condition (AO) in Proposition \ref{propE}. To see this, we again have to prove Lemma \ref{key lemma} for $U_g^\alpha e$, but it is straightforward since we have already obtained a concrete decomposition of $U_g^\alpha$ in the last part of the proof of Lemma \ref{key lemma}. It is enough to add the projection $e$ to the decomposition.

\appendix 
\section{Solidity for actions on von Neumann algebras}
\label{Solidity for actions on von Neumann algebras}

Throughout this appendix, we fix a free ultrafilter $\omega$ on $\N$. We use \textit{Ocneanu's ultraproduct von Neumann algebras} $M^\omega$ in this section and we refer the reader to \cite{AH12}.

\subsection{Definition and characterization}
\label{Definition and characterization}

Recall that a diffuse von Neumann algebra $M$ with separable predual is \textit{solid} if for any diffuse von Neumann subalgebra $A\subset M$ with expectation, the relative commutant $A'\cap M$ is amenable \cite{Oz03}. Any non-amenable solid factor is \textit{full} (i.e.\ $M'\cap M^\omega =\C$) \cite[Proposition 7]{Oz03} \cite[Theorem 3.1]{HU15}. In particular, any non-amenable subfactor (with expectation) of a solid von Neumann algebra is full.

Let $B$ be a diffuse von Neumann algebra with separable predual, $ G $ a countable discrete group, $ G  \act^\alpha B$ an action. Assume $\alpha$ is \textit{free}, that is, $B' \cap M\subset B$ where $M:=B\rtimes_\alpha  G $. 
We say that $\alpha$ is \textit{solid} if for any diffuse von Neumann subalgebra $A\subset B$ with expectation, the relative commutant $A'\cap M$ is amenable. The goal of this appendix is to prove the following theorem. This is a generalization of \cite[Proposition 6]{CI08}.

\begin{thm}\label{solid-thm}
	Let $B, G ,\alpha,M$ be as above.
\begin{enumerate}
	\item If $\alpha$ is solid, then for any intermediate von Neumann algebra $B\subset N\subset M$ with expectation, there exist mutually orthogonal projections $\{z_n\}_{n\geq 0}$ in $\mathcal Z(N)\subset \mathcal Z(B)$ such that $\sum_n z_n=1$, $Nz_0$ is amenable, and $(N'\cap B^\omega)z_n=\C z_n$ for all $n\geq 1$.

	\item Assume that $B$ is commutative and write $B=L^\infty(X)$ and $L(\mathcal R) = M$, where $\mathcal R$ is the equivalence relation arising from $\alpha$. Then $\alpha$ is solid if and only if for any subequivalence relation $\mathcal S \subset \mathcal R$, there exists a partition $\{X_n\}_{n\geq 0}$ of $X$ into $\mathcal S$ invariant measurable sets such that 
\begin{itemize}
	\item[$\rm(a)$] $\mathcal S|_{X_0}$ is hyperfinite; and 
	\item[$\rm(b)$] $\mathcal S|_{X_n}$ is strongly ergodic for all $n\geq 1$.
\end{itemize}
\end{enumerate}
\end{thm}

\subsection{Proof of Theorem \ref{solid-thm}}

We keep the following setting. Let $B\subset^{E_B} N\subset^{E_N} M$ be inclusions of von Neumann algebras with separable predual and with expectation. We write $E_B\circ E_N$ again as $E_B$ and take a faithful normal state $\varphi\in M_*$ such that $\varphi = \varphi \circ E_B = \varphi \circ E_N$. We have inclusions with expectation
	$$N'\cap (B^\omega)_{\varphi^\omega} \subset  N'\cap B^\omega \subset M^\omega.$$
To prove Theorem \ref{solid-thm}, we have to recall the following three lemmas. We introduce slightly more general ones, but the same proofs work by using the fact that $B$ is contained in $N$.

The first one is by Ando and Haagerup \cite[Theorem 5.2]{AH12} (see also \cite[Lemma 2.5]{HR14}).

\begin{lem}\label{AH's lemma}
	If $N'\cap (B^\omega)_{\varphi^\omega}=\C$, then $N'\cap B^\omega =\C$.
\end{lem}

We next introduce Ioana's lemma \cite[Lemma 2.7]{Io12}. For the proof, see its generalization \cite[Theorem 2.3]{HR14} by Houdayer and Raum and use the above Lemma \ref{AH's lemma}.

\begin{lem}\label{Ioana's lemma}
	If there is a projection $z\in \mathcal{Z}(N'\cap B^\omega)$ such that $(N'\cap B^\omega) e$ is discrete, then $e\in \mathcal{Z}(N'\cap B)$ and $(N'\cap B^\omega) e=(N'\cap B)e$. 
\end{lem}

The last lemma is due to Popa \cite[Proposition 7]{Oz03}. For the proof, see its generalization \cite[Theorem 3.1]{HU15} by Houdayer and Raum and use Lemma \ref{AH's lemma}.

\begin{lem}\label{Popa's lemma}
	The following conditions are equivalent.
\begin{enumerate}
	\item $N'\cap B^\omega$ is diffuse.
	\item There exists a decreasing sequence $\{A_{n}\}_{n\geq 0}$ of diffuse abelian von Neumann subalgebras in $B$ with expectation such that $N=\bigvee_n (A_n'\cap N)$.
\end{enumerate}
\end{lem}

Summarizing these lemmas, we get the following useful proposition.

\begin{prop}\label{conclusion-prop}
	Let $B\subset N\subset M$ be as above. Take the unique projection $z\in \mathcal Z(N'\cap B^\omega)$ such that $(N'\cap B^\omega )z$ is discrete and $(N'\cap B^\omega )z^\perp$ is diffuse. Then we have:
\begin{itemize}
	\item $z\in \mathcal Z(N'\cap B)\subset \mathcal Z(B)$ and $(N'\cap B^\omega) z=(N'\cap B) z$;

	\item there is a decreasing sequence $\{A_{n}\}_{n\geq 0}$ of diffuse abelian von Neumann subalgebras in $B z^\perp$ with expectation such that $Nz^\perp =\bigvee_n (A_n'\cap Nz^\perp)$.
\end{itemize}
\end{prop}

\begin{proof}[Proof of Theorem \ref{solid-thm}]
	1. Observe first that for any projection $p\in \mathcal Z(B)$ and any diffuse von Neumann subalgebra $A\subset Bp$ with expectation, the relative commutant $A'\cap pMp$ is amenable (indeed $(A\oplus Bp^\perp)'\cap M = (A'\cap pMp)\oplus \mathcal Z(B)p^\perp$ is amenable). 

We see that for any projection $p\in \mathcal Z(B)$ and any intermediate von Neumann algebra $Bp\subset Q\subset pMp$ with expectation, if $Q$ has no amenable direct summand, then $Q'\cap (B^\omega p)$ is discrete. Indeed, if not, then we can apply Proposition \ref{conclusion-prop} to $Bp\subset Q\subset pMp$, and there is $p\neq z\in \mathcal Z(Q'\cap Bp)\subset \mathcal Z(B)$ and $\{A_n\}_{n\geq 0}$ such that $Qz^\perp = \bigvee_n A_n'\cap Qz^\perp$. Since $Qz^\perp$ is not amenable, $A_n'\cap Qz^\perp$ is not amenable for $n$ large enough. This implies that $A_n\subset Bz^\perp$ is diffuse while $A_n'\cap z^\perp Mz^\perp $ is not amenable. This contradicts the solidity of $\alpha$.

Let $B\subset N\subset M$ be as in the statement and take the unique projection $z_0\in \mathcal Z(N)\subset \mathcal Z(B)$ such that $Nz_0$ is amenable and $Nz_0^\perp$ has no amenable direct summand. We may assume $z_0^\perp \neq 0$. Put $p:=z_0^\perp$ and $Q:=Np$. By the observation in the previous paragraph, $Q'\cap (B^\omega p)$ is discrete, hence $Q'\cap (B^\omega p) = Q'\cap Bp = \mathcal Z(Q)$ by Proposition \ref{conclusion-prop}. Let $\{z_n\}_{n\geq 1}$ be a family of mutually orthogonal minimal projections in $\mathcal Z(Q)$ such that $z_0^\perp = \sum_{n\geq 1}z_n$. Then since $z_n$ is minimal in $Q'\cap (B^\omega p)$, the conclusion follows.

	2. We have only to prove the `if' direction. For this we can follow the proof of (3)$\Rightarrow$(2)$\Rightarrow$(1) in \cite[Proposition 6]{CI08}.
\end{proof}

\small{

}






\begin{thebibliography}{BKV19}

\bibitem[AH12]{AH12} H. Ando and U. Haagerup, \textit{Ultraproducts of von Neumann algebras.} J. Funct. Anal. {\bf 266} (2014), 6842--6913.

\bibitem[AIM19]{AIM19} Y. Arano, Y. Isono and A. Marrakchi, \textit{Ergodic theory of affine isometric actions on Hilbert spaces.} Preprint 2019. \texttt{arXiv:1911.04272}


\bibitem[BR97]{BR97} O. Bratteli and D. W. Robinson, \textit{Operator algebras and quantum statistical mechanics. 2. Equilibrium states. Models in quantum statistical mechanics. Second edition.} Texts and Monographs in Physics. Springer-Verlag, Berlin, 1997. xiv+519 pp.

\bibitem[BO08]{BO08} N. P. Brown and N. Ozawa, \textit{$C^{\ast}$-algebras and finite-dimensional approximations}. Graduate Studies in Mathematics, 88. American Mathematical Society, Providence, RI, 2008.

\bibitem[CI08]{CI08} I. Chifan and A. Ioana, \textit{Ergodic subequivalence relations induced by a Bernoulli action.} Geom. Funct. Anal. {\bf 20} (2010), no. 1, 53--67.

\bibitem[Co75]{Co75} A. Connes, \textit{Classification of injective factors.} Ann. of Math. (2) {\bf 104} (1976), no. 1, 73--115.

\bibitem[EK98]{EK98} D. Evans and Y. Kawahigashi, \textit{Quantum symmetries on operator algebras.} Oxford Mathematical Monographs. Oxford Science Publications. The Clarendon Press, Oxford University Press, New York, 1998. xvi+829 pp.

\bibitem[Ha81]{Ha81} T. Hamachi, \textit{On a Bernoulli shift with non-identical factor measures.} Ergodic Theory Dynam. Systems {\bf 1} (1981), 273--284.

\bibitem[Hi02]{Hi02} F. Hiai, \textit{$q$-deformed Araki-Woods algebras.} Operator algebras and mathematical physics (Constanţa, 2001), 169--202, Theta, Bucharest, 2003.

\bibitem[HR14]{HR14} C. Houdayer and S. Raum, \textit{Asymptotic structure of free Araki--Woods factors.} Math. Ann. {\bf 363} (2015), no. 1-2, 237--267.

\bibitem[HU15]{HU15} C. Houdayer and Y. Ueda, \textit{Rigidity of free product von Neumann algebras.} Compos. Math. {\bf 152} (2016), no. 12, 2461--2492.

\bibitem[HI15a]{HI15a} C. Houdayer and Y. Isono, \textit{Unique prime factorization and bicentralizer problem for a class of type III factors.} Adv. Math. {\bf 305 } (2017), 402--455.

\bibitem[HI15b]{HI15b} C. Houdayer and Y. Isono, \textit{Bi-exact groups, strongly ergodic actions and group measure space type III factors with no central sequence.} Comm. Math. Phys. {\bf 348} (2016), no. 3, 991--1015.

\bibitem[Io12]{Io12} A.~Ioana, \textit{Cartan subalgebras of amalgamated free product $\rm II_1$ factors} (with an appendix joint with S.~Vaes). Ann.\ Sci.\ \'{E}cole Norm.\ Sup. {\bf 48} (2015), 71--130.

\bibitem[Io17]{Io17} A. Ioana, \textit{Rigidity for von Neumann algebras.} Proceedings of the International Congress of Mathematicians—Rio de Janeiro 2018. Vol. III. Invited lectures, 1639--1672, World Sci. Publ., Hackensack, NJ, 2018.



\bibitem[Ka48]{Ka48} S. Kakutani, \textit{On equivalence of infinite product measures.} Ann. of Math. {\bf 49} (1948), 214--224.

\bibitem[Ko09]{Ko09} Z. Kosloff, \textit{On a type III$_1$ Bernoulli shift.} Ergodic Theory Dynam. Systems {\bf 31} (2011), 1727--1743.

\bibitem[Kr70]{Kr70} U. Krengel, \textit{Transformations without finite invariant measure have finite strong generators.} In Contributions to Ergodic Theory and Probability (Proc. Conf., Ohio State Univ., Columbus, Ohio, 1970), Springer, Berlin, pp. 133--157.

\bibitem[Ma16]{Ma16} A. Marrakchi, \textit{Solidity of type III Bernoulli crossed products.} Comm. Math. Phys. {\bf 350} (2017), no. 3, 897--916.


\bibitem[Oz03]{Oz03} N. Ozawa, \textit{Solid von Neumann algebras}. Acta Math. {\bf 192} (2004), 111--117.
\bibitem[Oz04]{Oz04} N. Ozawa, \textit{A Kurosh type theorem for type $\rm {II}_{1}$ factors}. Int. Math. Res. Not. (2006), Art. ID 97560, 21 pp.
\bibitem[Oz06]{Oz06} N. Ozawa, \textit{Amenable actions and applications.} International Congress of Mathematicians. Vol. II, 1563--1580, Eur. Math. Soc., Zürich, 2006.
\bibitem[Oz08]{Oz08} N. Ozawa, \textit{An example of a solid von Neumann algebra.} Hokkaido Math. J. {\bf 38 }(2009), no. 3, 557--561.




\bibitem[Po06]{Po06} S. Popa, \textit{Deformation and rigidity for group actions and von Neumann algebras.} In Proceedings of the International Congress of Mathematicians (Madrid, 2006), Vol. I, European Mathematical Society Publishing House, 2007, p. 445--477.



\bibitem[Ta03]{Ta03} M.~Takesaki, \textit{Theory of operator algebras. $\rm II$}. 
Encyclopaedia of Mathematical Sciences, 125. Operator Algebras and Non-commutative Geometry, 6. Springer-Verlag, Berlin, 2003. xxii+518 pp.

\bibitem[VW17]{VW17} S. Vaes and J. Wahl, \textit{Bernoulli actions of type $\rm III_1$ and $L^2$-cohomology.} Geom. Funct. Anal. {\bf 28} (2018), no. 2, 518--562.

\bibitem[Va10]{Va10} S. Vaes, \textit{Rigidity for von Neumann algebras and their invariants.} Proceedings of the International Congress of Mathematicians. Volume III, 1624--1650, Hindustan Book Agency, New Delhi, 2010.

\bibitem[BKV19]{BKV19} M. Bj$\rm \ddot{o}$rklund, Z. Kosloff and S. Vaes, \textit{Ergodicity and type of nonsingular Bernoulli actions.} 
Preprint 2019.  \texttt{arXiv:1901.05723}

\end{thebibliography}
\end{document}